\documentclass{amsart}
\usepackage{amsmath}
  \usepackage{graphics} 
  \usepackage{epsfig} 
\usepackage{graphicx}
\usepackage{epstopdf}

\newtheorem{theorem}{Theorem}[section]

\newtheorem{proposition}{Proposition}

\theoremstyle{definition}
\newtheorem{definition}[theorem]{Definition}
\newtheorem{remark}{Remark}

\newcommand{\e}{\varepsilon}
\newcommand{\ds}{\displaystyle}

\def\XXint#1#2#3{{\setbox0=\hbox{$#1{#2#3}{\int}$ }
\vcenter{\hbox{$#2#3$ }}\kern-.6\wd0}}

\def\div{\,\mathrm{div}\,}

\title[OCP for Linear Elliptic Equations]
      {Optimal $L^2$-Control Problem in Coefficients\\ for a Linear Elliptic Equation.\\
      II. Approximation of Solutions and Optimality Conditions}

\author[T.~Horsin and P.~I.~Kogut and O.~Wilk]{}


 \email{p.kogut@i.ua}
 \email{thierry.horsin@lecnam.net}
 \email{olivier.wilk@cnam.fr}

\thanks{ }

\begin{document}
\maketitle

\centerline{\scshape Thierry Horsin}
\medskip
{\footnotesize
 \centerline{ Conservatoire National des Arts et M\'{e}tiers,}
   \centerline{M2N, Case 2D 5000,}
   \centerline{292 rue Saint-Martin, 75003 Paris, France}
}

\medskip

\centerline{\scshape Peter I. Kogut }
\medskip
{\footnotesize
 \centerline{Department of Differential Equations,}
   \centerline{Dnipropetrovsk National University,}
   \centerline{ Gagarin av., 72, 49010 Dnipropetrovsk, Ukraine}
} 

\medskip

\centerline{\scshape Olivier Wilk}
\medskip
{\footnotesize
 \centerline{ Conservatoire National des Arts et M\'{e}tiers,}
   \centerline{M2N, Case 2D 5000,}
   \centerline{292 rue Saint-Martin, 75003 Paris, France}
}

\bigskip


\begin{abstract}
In this paper we study we study a Dirichlet optimal control problem
associated with a linear elliptic equation the coefficients of which
we take as controls in the class of integrable functions. The characteristic feature of this control object is the fact that the skew-symmetric part of matrix-valued control $A(x)$ belongs to $L^2$-space (rather than $L^\infty)$. In spite of the fact that the equations of this type can exhibit non-uniqueness of weak solutions, the corresponding OCP, under rather general assumptions on the class of admissible controls, is well-posed and admits a nonempty set of solutions \cite{HK_1}. However, the optimal solutions to such problem may have a singular character.
We show that some of optimal solutions can be attainable by solutions of special optimal control problems in perforated domains with fictitious boundary controls on the holes.
\end{abstract}

In this paper we deal with the following optimal control problem (OCP) in coefficients for a linear elliptic  equation
\begin{equation}\label{0.1}
\left\{
\begin{array}{c}
\ds\text{Minimize } I(A,y)=\left\|y-y_d\right\|^2_{L^2(\Omega)}
 + \int_\Omega\left(\nabla y, A^{sym}\nabla y\right)_{\mathbb{R}^N}\,dx\\
 \text{subject to the constraints }\\
-\div\big(A^{sym}\nabla y+A^{skew}\nabla y\big) = f\quad\text{ in }\Omega,\\
y=0\quad\text{ on }\partial\Omega\\
A\in \mathfrak{A}_{ad},
\end{array}
\right.
\end{equation}
where  $(A^{sym},\,A^{skew})\in L^\infty(\Omega;R^{N\times N})\times L^2(\Omega;R^{N\times N})$ are respectively the symmetric and antisymmetric part of the control $A$, $y_d\in L^2(\Omega)$ and $f\in H^{-1}(\Omega)$ are  given distributions, and $\mathfrak{A}_{Ad}$ denotes the class of admissible controls which will be precised later.

The characteristic feature of this problem is the fact that the skew-symmetric part of matrix $A(x)$  belongs to $L^2$-space (rather than $L^\infty)$. As a result, the existence and uniqueness of the weak solutions to the corresponding boundary value problem \eqref{0.1} { are }usually  drastically different from the properties of solutions to the elliptic equations with $L^\infty$-matrices in coefficients. In most of the cases, the situation can deeply change for the matrices $A$ with unremovable singularity. As a rule, some of the weak solutions  can be attained  by the weak solutions to the similar boundary value problems with $L^\infty$-approximated matrix $A$. However, this type does not exhaust all weak solutions to the above problem. There is another type of weak solutions called non-variational \cite{Zh_97,Zhik1_04}, singular \cite{Buttazzo_Kogut,Kogut1,Kogut2,Zuazua1}, pathological \cite{Maz,Serin} and others. As for the optimal control problem \eqref{0.1} we have the following result \cite{HK_1} (see \cite{HK1} for comparison): for any approximation $\left\{A^\ast_k\right\}_{k\in \mathbb{N}}$ of the  matrix $A^\ast\in L^2\big(\Omega;\mathbb{S}^N_{skew}\big)$ with properties $\left\{A^\ast_k\right\}_{k\in \mathbb{N}}\subset L^\infty(\Omega;\mathbb{S}^N_{skew})$ and
$A^\ast_k\rightarrow A^\ast$ strongly in $L^2(\Omega;\mathbb{S}^N_{skew})$, optimal solutions to the corresponding regularized OCPs  associated with matrices $A^\ast_k$  always lead in the limit as $k\to\infty$ to some admissible (but not optimal in general) solution $(\widehat{A},\widehat{y}\,)$ of the original OCP \eqref{0.1}. Moreover, this limit pair can depend on the choice of the approximative sequence $\left\{A^\ast_k\right\}_{k\in \mathbb{N}}$.  However, as follows from counter-example, given in \cite{HK_1}, it is possible a situation when none of optimal solutions to OCP \eqref{0.1} can be attainable in such way.
Therefore, the aim of this paper is to discuss a scheme of approximation for OCP \eqref{0.1} in order to attain the other types of optimal solutions, and derive the first order optimality system to this problem.

In order to illustrate the difficulties on the approximations of the OCPs due to the possible existence of variational and non-variational solutions, we present some numerical simulations in section \ref{NumSim}.

In section \ref{Sec 5} we give a precise description of the class of admissible
controls $\mathfrak{A}_{ad}\subset L^2\big(\Omega;\mathbb{R}^{N\times N}\big)$
which guarantee that non-variational solutions can be attained through the
sequence of optimal solutions to OCPs in special perforated domains with
fictitious boundary controls on the boundary of holes. Namely, we consider the
following family of regularized OCPs
\begin{equation}\label{0.2}
\left\{
\begin{array}{c}
\ds\text{Minimize } I_\e(A,v,y):=\left\|y-y_d\right\|^2_{L^2(\Omega_\e)}+ \int_{\Omega_\e}\left(\nabla y, A^{sym}\nabla y\right)_{\mathbb{R}^N}\,dx\\
\ds  + \frac{1}{\e^\sigma}\|v\|^2_{H^{-\frac{1}{2}}(\Gamma_\e)}\\
 \text{subject to the constraints }\\
-\div\big(A^{sym}\nabla y+A^{skew}\nabla y\big) = f\quad\text{in }\ \Omega_\e,\\
y=0\text{ on }\partial\Omega,\quad
\partial y/ \partial \nu_{A}=v\ \text{on }\Gamma_\e,\\
\ y\in H^1_0(\Omega_\e;\partial\Omega),
\end{array}
\right.
\end{equation}
where $\Omega_\e$ is the subset of $\Omega$ such that $\partial\Omega\subset\partial\Omega_\e$, $\sigma>0$, and $\|A(x)\|_{\mathbb{S}^N}:=\max_{i,j=1,\dots,N}\left|a_{ij}(x)\right|\le\e^{-1}$ a.e. in $\Omega_\e$.
Here, $v$ stands for the fictitious control.

We show that OCP \eqref{0.2} has a nonempty set of solutions $(A_\e^0,v_\e^0,y_\e^0)$ for every $\e>0$.
Moreover, as follows from \eqref{0.2}$_1$, the cost functional
$I_\e$ seems to be rather sensitive with respect to the fictitious
controls. Due to this fact,
we prove that  the sequence $\left\{(A_\e^0,y_\e^0)\right\}_{\e>0}$ gives in the limit an optimal solution $(A^0,y^0)$ to the
original problem.

The main technical difficulty, which is related with the study of the asymptotic behaviour of OCPs \eqref{0.2} as $\e\to 0$,
deals with the identification of the limit
$\lim_{\e\to 0}\left\{\left<v^0_\e,y^0_{\e}\right>_{H^{-\frac{1}{2}}(\Gamma_\e);H^{\frac{1}{2}}(\Gamma_\e)}\right\}_{\e>0}$
of two weakly convergent sequences. Due to the special properties of the skew-symmetric parts of admissible controls $A\in \mathfrak{A}_{ad}\subset L^2\big(\Omega;\mathbb{S}^N\big)$, we show that this limit can be recovered in an explicit form.
We also show in this section that
the energy equalities to the regularized boundary value problems can be specified by two extra terms which characterize the presence of the-called hidden singular energy coming from $L^2$-properties of skew-symmetric components $A^{skew}$ of admissible controls.

In conclusion,  in Section~\ref{Sec 6}, we derive the optimality conditions for regularized OCPs \eqref{0.2} and show that the limit passage in optimality system for the regularized problems \eqref{0.2} as $\e\to 0$ leads to the optimality system for the original OCP \eqref{0.1}.


\section{Notation and Preliminaries}\label{Sec_1}
\label{Sec 1}

Let $\Omega$ be a bounded open connected subset of $\mathbb{R}^N$ ($N\ge 2$)
with Lipschitz boundary $\partial\Omega$.
By $H_0^1(\Omega)$ we denote the closure of $C_0^\infty
(\Omega)$-functions in
the Sobolev space $H^{1}(\Omega)$, while $H^{-1}(\Omega)$ { denotes
the dual of $H^{1}_0(\Omega)$.
Let $\Gamma$ be a part of the boundary
$\partial\Omega$ with positive $(N-1)$-dimensional measures. We consider
$C^\infty_0(\mathbb{R}^N;\Gamma)=\left\{\varphi\in
C^\infty_0(\mathbb{R}^N)\ :\ \varphi=0\text{ on }\Gamma\right\}$,
and denote  $H^{1}_0(\Omega;\Gamma)$ its  closure   with respect to the norm
$\|y\|=\left(\int_\Omega \|\nabla y\|^2_{\mathbb{R}^N}\,dx\right)^{1/2}.$

Let $\mathbb{M}^N=\mathbb{S}^N_{sym}\oplus\mathbb{S}^N_{skew}$ be the set of all $N\times N$ {real} matrices.
Here, $\mathbb{S}^N_{skew}$ stands for the  set of all
skew-symmetric matrices $C=[c_{ij}]_{i,j=1}^N$, whereas $\mathbb{S}^N_{sym}$ is the set of all $N\times N$ symmetric matrices.

Let
$L^2(\Omega)^{\frac{N(N-1)}{2}}=L^2\big(\Omega;\mathbb{S}^N_{skew}\big)$
be the {normed} space of measurable square-integrable functions whose values are skew-sym\-met\-ric matrices.
By analogy, we can define the space
$L^2(\Omega)^{\frac{N(N+1)}{2}}=L^2\big(\Omega;\mathbb{S}^N_{sym}\big)$.

Let $A(x)$ and $B(x)$ be given  matrices such that $A,B\in L^2(\Omega;\mathbb{S}^N_{skew})$. We say that these matrices are related by the binary relation $\preceq$ on the set $L^2(\Omega;\mathbb{S}^N_{skew})$ (in symbols, $A(x)\preceq B(x)$ a.e. in $\Omega$), if
\begin{equation}
\label{1.0b}
\mathcal{L}^N \left\{\bigcup_{i=1}^N \bigcup_{j=i+1}^N\left\{ x\in\Omega\ :\ | a_{ij}(x) |>  |b_{ij}(x)|\right\}\right\}=0.
\end{equation}
Here, $\mathcal{L}^N(E)$ denotes the $N$-dimensional Lebesgue measure of $E\subset \mathbb{R}^N$  defined on the completed borelian $\sigma$-algebra.

We define the divergence $\div A$ of a matrix $A\in L^2\big(\Omega;\mathbb{M}^N\big)$ as a vector-valued distribution
$d\in H^{-1}(\Omega;\mathbb{R}^N)$ by the following rule
\begin{equation}
\label{1.0a}
\left<d_i,\varphi\right>_{H^{-1}(\Omega);H^1_0(\Omega)}= -\int_\Omega
(\mathbf{a}^t_i,\nabla\varphi)_{\mathbb{R}^N}\,dx,\quad \forall\,\varphi\in
C^\infty_0(\Omega),\quad\forall\,i\in\left\{1,\dots,N\right\},
\end{equation}
where $\mathbf{a}_i$ stands for the $i$-th row of the matrix $A$.

For fixed two constants $\alpha$ and $\beta$ such that
$0<\alpha\le\beta<+\infty$,
we define $\mathfrak{M}_\alpha^\beta(\Omega)$ as a set of all matrices
$A=[a_{i\,j}\,]$ in $L^\infty(\Omega;\mathbb{S}^N_{sym})$ such that
\begin{equation}
\label{1.1}
\alpha \|\xi\|^2_{\mathbb{R}^N}\le  \left(A\xi,\xi\right)_{\mathbb{R}^N}\le \beta \|\xi\|^2_{\mathbb{R}^N},\quad \text{a.e. in } \Omega,\quad \forall\,\xi\in \mathbb{R}^{N}.
\end{equation}

Let $A\in L^2\big(\Omega;\mathbb{M}^N\big)$ be an arbitrary matrix. In view of the representation $A=A^{sym}+A^{skew}$, we can associate with $A$ the form $\varphi(\cdot,\cdot)_A:H^1_0(\Omega)\times H^1_0(\Omega)\to \mathbb{R}$ following the rule
\[
\varphi(y,v)_A= \int_\Omega \big(\nabla v,A^{skew}(x)\nabla y\big)_{\mathbb{R}^N}\,dx,\quad \forall\,y,v\in H^1_0(\Omega).
\]
By analogy with \cite{HK_1}, we introduce the following concept.
\begin{definition}
\label{Def 2.7}  We say that an element $y\in H^1_0(\Omega)$ belongs to the set $D(A)$ if
\begin{equation}
\label{2.8}
\left|\int_\Omega \big(\nabla \varphi,A^{skew}\nabla y\big)_{\mathbb{R}^N}\,dx\right|\le c(y,A^{skew}) \left(\int_\Omega |\nabla \varphi|^2_{\mathbb{R}^N}\,dx\right)^{1/2},\  \forall\,\varphi\in C^\infty_0(\Omega)
\end{equation}
with some constant $c$ depending only of $y$ and $A^{skew}$.
\end{definition}

As a result, having set
\[
[y,\varphi]_A= \int_\Omega \big(\nabla \varphi,A^{skew}(x)\nabla y\big)_{\mathbb{R}^N}\,dx,\quad \forall\,y\in D(A),\ \forall\,\varphi\in C^\infty_0(\Omega),
\]
we see that the bilinear form $[y,\varphi]_A$ can be defined for all $\varphi\in H^1_0(\Omega)$ using \eqref{2.8} and the standard rule
\begin{equation}
\label{2.8a}
[y,\varphi]_A=\lim_{\e\to 0}\,[y,\varphi_\e]_A,
\end{equation}
where $\left\{\varphi_\e\right\}_{\e>0}\subset C^\infty_0(\Omega)$ and $\varphi_\e\rightarrow \varphi$ strongly in $H^1_0(\Omega)$.

Let $\e$ be a small parameter, $I_\e:\mathbb{U}_\e\times
\mathbb{Y}_\e\rightarrow \overline{\mathbb{R}}$ be a cost functional, $\mathbb{Y}_\e$ be a space of states,
and $\mathbb{U}_\e$ be a space of controls. Let
$$
\Xi_\e\subset\left\{(u_\e,y_\e)\in \mathbb{U}_\e\times
\mathbb{Y}_\e \ :\ u_\e\in U_\e,\ I_\e(u_\e,y_\e)<+\infty\right\}
$$
be a set of all admissible pairs linked by some state equation.
We consider the following constrained minimization problem:
\begin{equation}
\label{1.3} (\mathrm{CMP_\e})\ :\qquad\qquad \left<
\inf\limits_{(u,y)\in\,\Xi_\e} I_\e(u,y)\right>.
\end{equation}
Since the sequence of
constrained minimization problems \eqref{1.3} lives in variable spaces
$\mathbb{U}_\e\times\mathbb{Y}_\e$, we assume that there exists a Banach
space $\mathbb{U}\times \mathbb{Y}$ with respect to
which a convergence in the scale of spaces $\left\{\mathbb{U}_\e\times \mathbb{Y}_\e\right\}_{\e>0}$ is defined
(for the details, we refer to \cite{KL,Zh_98}). In the sequel, we use the following notation for this convergence
$(u_\e,y_\e)\stackrel{\mu}{\longrightarrow}\, (u,y)$  in $\mathbb{U}_\e\times \mathbb{Y}_\e$.

In order to study the asymptotic behavior of a family of $(\mathrm{CMP_\e})$,
the passage to the limit in \eqref{1.3} as the small parameter $\e$ tends to zero has to be realized.
Following the scheme of the direct variational convergence \cite{KL},
we adopt the following definition for the convergence of minimization problems in variable spaces.

\begin{definition}
\label{Def 1.4} A problem $\left<\inf_{(u,y)\in\Xi} I(u,y)\right>$ is
the variational limit of the sequence (\ref{1.3}) as $\e\to 0$
\[
\left(\text{in symbols, }\quad
\left<\inf\limits_{(u,y)\in\,\Xi_\e} I_\e(u,y)\right>\,\stackrel{\text{Var}}{\xrightarrow[\e\to 0]{}}\,
\left<\inf_{(u,y)\in\Xi} I(u,y)\right>\ \right)
\]
if
and only if the following conditions are satisfied:
\begin{enumerate}
\item[$\quad$(d)] If sequences $\left\{\e_k\right\}_{k\in \mathbb{N}}$ and
$\left\{(u_k,y_k)\right\}_{k\in \mathbb{N}}$ are such that $\varepsilon_k
\rightarrow 0$ as $k\rightarrow \infty$, $(u_k,y_k)\in
\Xi_{\varepsilon_k}$ $\forall\,k\in \mathbb{N}$, and
$(u_k,y_k)\stackrel{\mu}{\longrightarrow}\,(u,y)$ in $\mathbb{U}_{\varepsilon_k}\times\mathbb{Y}_{\varepsilon_k}$, then
\begin{equation}
\label{1.6} (u,y)\in \Xi;\quad I(u,y)\le
\liminf_{k\to\infty}I_{\e_k}(u_k,y_k).
\end{equation}

\item[$\quad$(dd)] For every $(u,y) \in \Xi\subset \mathbb{U}\times\mathbb{Y}$ there are a constant $\varepsilon^0 >0$ and a
sequence
$\left\{(u_\e,y_\e)\right\}_{\varepsilon>0}$ (called a $\Gamma$-realizing sequence) such that
\begin{gather}
(u_\e,y_\e)\in\Xi_\varepsilon,\ \forall\,\varepsilon \leq
  \varepsilon^0,\quad
(u_\e,y_\e)\,\stackrel{\mu}{\longrightarrow}\, ({u},{y})\
\text{ in }\ \mathbb{U}_\e\times \mathbb{Y}_\e,\label{1.7a}\\
\label{1.7c} I(u,y)\ge \limsup_{\varepsilon\to 0}
I_{\varepsilon}(u_\e,y_\e).
\end{gather}
\end{enumerate}
\end{definition}
\begin{theorem}[\cite{KL}]
\label{Th 1.8} Assume that the constrained minimization problem
\begin{equation}
\label{1.9} \Big \langle \inf_{(u,y) \in \Xi_0} I_0(u,y)\Big
\rangle
\end{equation}
is the variational limit of sequence (\ref{1.3}) in the sense
of Definition \ref{Def 1.4} and this problem has a nonempty set of
solutions
$$
\Xi_0^{opt}:=\left\{(u^0,y^0)\in\Xi_0\ :\ I_0(u^0,y^0)=\inf_{(u,y) \in \Xi_0} I_0(u,y)\right\}.
$$
For every $\varepsilon >0$, let
$(u^0_\varepsilon,y^0_\e)\in \Xi_\varepsilon$ be a minimizer of $I_\e$ on the
corresponding set $\Xi_\e$. If the sequence
$\{(u^0_\varepsilon,y^0_\e)\}_{\varepsilon >0}$ is relatively compact with
respect to the $\mu$-convergence in variable spaces $\mathbb{U}_\e\times \mathbb{Y}_\e$, then
there exists a pair $(u^0,y^0)\in \Xi_0^{opt}$ such that
\begin{gather}
\label{1.10} (u^0_\varepsilon,y^0_\e)\,\stackrel{\mu}{\longrightarrow}\,
(u^0,y^0)\quad \text{in }\ %
\mathbb{U}_\e\times \mathbb{Y}_\e,\\
\inf_{(u,y)\in\,\Xi_0}I_0(u,y)= I_0\left(u^0,y^0\right) =\lim_{\varepsilon\to
0} I_{\varepsilon}(u^0_\varepsilon,y^0_\varepsilon) =\lim_{\varepsilon\to
0}\inf_{(u_\varepsilon,y_\e)\in\,{\Xi}_\varepsilon}
{I}_\varepsilon(u_\varepsilon,y_\e).\label{1.11}
\end{gather}
\end{theorem}


\section{Setting of the Optimal Control Problem}
\label{Sec 2}

Let $f\in H^{-1}(\Omega)$ and $y_d\in L^2(\Omega)$ be given distributions.

by choosing an appropriate control $A\in L^2(\Omega;\mathbb{M}^N)$.

More precisely, we are
concerned with the following OCP
\begin{gather}
\label{2.3}
\text{Minimize } I(A,y)=\left\|y-y_d\right\|^2_{L^2(\Omega)}+\int_\Omega\left(\nabla y, A^{sym}\nabla y\right)_{\mathbb{R}^N}\,dx
\end{gather}
subject to the constraints
\begin{gather}
\label{2.1}
-\div\big(A(x)\nabla y\big) = f\quad\text{in }\ \Omega,\\
\label{2.2}
y=0\text{ on }\partial\Omega,\\
\label{2.3a}
A\in \mathfrak{A}_{ad}.
\end{gather}
To define the class of admissible controls $\mathfrak{A}_{ad}$, , we introduce the following sets.
\begin{align}
\label{2.3b} U_{a,1}&=\left\{\left. A=[a_{i\,j}]\in L^1(\Omega;\mathbb{S}^N_{sym})\
\right| TV(a_{ij})\le c,\ 1\le i\le j\le N\right\},\\[1ex]
\label{2.3bb} U_{b,1}&=\left\{\left. A=[a_{i\,j}]\in L^\infty(\Omega;\mathbb{S}^N_{sym})\
\right| A\in \mathfrak{M}_\alpha^\beta(\Omega)\right\},\\[1ex]
\label{2.3c} U_{a,2}&=\left\{\left. A=[a_{i\,j}]\in L^2(\Omega;\mathbb{S}^N_{skew})\ \right|\ A(x)\preceq A^\ast(x)\ \text{a.e. in }\ \Omega\right\},\\
\label{2.3cc} U_{b,2}&=\left\{\left. A=[a_{i\,j}]\in L^2(\Omega;\mathbb{S}^N_{skew})\ \right|\  A\in Q\right\},
\end{align}
where $A^\ast\in L^2(\Omega;\mathbb{S}^N_{skew})$ is a given matrix, $c$ is a positive constant, $Q$ is a nonempty convex compact subset of $L^2(\Omega;\mathbb{S}^N_{skew})$ such that the null matrix $A\equiv [0]$ belongs to $Q$, and
\begin{equation*}
TV(f):= \sup\Big\{\int_\Omega f\,(\nabla,\varphi)_{\mathbb{R}^N}\,dx\,  :
\varphi\in C^1_0(\Omega;\mathbb{R}^N),\
|\varphi(x)|\le 1\ \text{for}\ x\in \Omega\Big\},
\end{equation*}
\begin{definition}
\label{Def 2.3e}
We say that a matrix $A=A^{sym}+A^{skew}$ is an
admissible control to the Dirichlet boundary value problem \eqref{2.1}--\eqref{2.2} (in symbols, $A\in \mathfrak{A}_{ad}$)
if $A^{sym}\in \mathfrak{A}_{ad,1}:=U_{a,1}\cap U_{b,1}$ and $A^{skew}\in \mathfrak{A}_{ad,2}:=U_{a,2}\cap U_{b,2}$.
\end{definition}

We  have the following result.
\begin{proposition}[\cite{HK_1}]
\label{Prop 2.3f} The set $\mathfrak{A}_{ad}$ is nonempty, convex, and sequentially compact with respect to the strong topology of
$L^2(\Omega;\mathbb{M}^N)$.
\end{proposition}

The distinguishing feature of optimal control problem \eqref{2.3}--\eqref{2.3a} is the fact that the matrix-valued control $A\in \mathfrak{A}_{ad}$ is merely measurable and belongs to the space $L^2\big(\Omega;\mathbb{M}^N\big)$ (rather than the space of bounded matrices $L^\infty\big(\Omega;\mathbb{M}^N\big)$). The unboundedness of the skew-symmetric part of matrix $A\in \mathfrak{A}_{ad}$ can have a reflection in non-uniqueness of weak solutions to the corresponding boundary value problem. It means that there exists
a matrix $A\in L^2\big(\Omega;\mathbb{M}^N\big)$ such that the corresponding state $y\in H^1_0(\Omega)$ may be not unique.
\begin{definition}
\label{Def 2.13}
We say that  $(A,y)$ is an admissible pair to the OCP \eqref{2.3}--\eqref{2.3a} if $A\in \mathfrak{A}_{ad}\subset L^2\big(\Omega;\mathbb{M}^N\big)$,
$y\in D(A)\subset H^1_0(\Omega)$, and the pair $(A,y)$ is related by the integral identity
\begin{equation}
\label{2.5}
\int_\Omega \big(\nabla \varphi,A^{sym}\nabla y+A^{skew}\nabla y\big)_{\mathbb{R}^N}\,dx
=\left<f,\varphi\right>_{H^{-1}(\Omega);H^1_0(\Omega)},\quad\forall\,\varphi\in C^\infty_0(\Omega).
\end{equation}
\end{definition}

We
denote by $\Xi$ the set of all admissible pairs for
the OCP \eqref{2.3}--\eqref{2.3a}. Let $\tau$ be the topology on the set of admissible pairs
$\Xi\subset L^2\big(\Omega;\mathbb{M}^N\big)\times
H^{1}_0(\Omega)$
which we define as the product of the strong topology of
$L^2\big(\Omega;\mathbb{M}^N\big)$ and the weak topology of
$H^{1}_0(\Omega)$. We say that a pair $(A^0,y^0)\in L^2\big(\Omega;\mathbb{M}^N\big)\times D(A^0)$ is
optimal for problem \eqref{2.3}--\eqref{2.3a} if
\[
(A^0,y^0)\,\in\,\Xi\ \text{ and }\ I(A^0,y^0)=\inf_{(A,y)\in\,\Xi}
I(A,y).
\]
As immediately follows from \eqref{2.8a}, every weak solution $y\in D(A)$ to the problem \eqref{2.1}--\eqref{2.2} satisfies the energy equality
\begin{equation}
\label{2.12}
\int_\Omega \big(A^{sym}\nabla y,\nabla y\big)_{\mathbb{R}^N}\,dx + [y,y]_A
=\left<f,y\right>_{H^{-1}(\Omega);H^1_0(\Omega)},
\end{equation}
where the value $[y,y]_A$ may not of constant sign for all $y\in D(A)$. Hence, the energy equality \eqref{2.12} does not allow us to derive a reasonable a priory estimate in $H^1_0$-norm for the weak solutions (see \cite{HK_1}).

As was shown in \cite{HK_1}, OCP \eqref{2.3}--\eqref{2.3a} is always regular, i.e. $\Xi\ne\emptyset$, and moreover, for each $f\in
H^{-1}(\Omega)$  and $y_d\in L^2(\Omega)$, this problem admits at least one solution. However, the main point is that for any approximation $\left\{A^\ast_k\right\}_{k\in \mathbb{N}}$ of the  matrix $A^\ast\in L^2\big(\Omega;\mathbb{S}^N_{skew}\big)$ with properties $\left\{A^\ast_k\right\}_{k\in \mathbb{N}}\subset L^\infty(\Omega;\mathbb{S}^N_{skew})$ and
$A^\ast_k\rightarrow A^\ast$ strongly in $L^2(\Omega;\mathbb{S}^N_{skew})$, optimal solutions to the corresponding regularized OCPs  associated with matrices $A^\ast_k$  always lead in the $\tau$-limit as $k\to\infty$ to some admissible (but not optimal in general) solution $(\widehat{A},\widehat{y}\,)$ of the original OCP \eqref{2.3}--\eqref{2.3a}. Moreover, this limit pair can depend on the choice of the approximative sequence $\left\{A^\ast_k\right\}_{k\in \mathbb{N}}$.  However, as follows from counter-example, given in \cite{HK_1}, it is possible a situation when none of optimal solutions to OCP \eqref{2.3}--\eqref{2.3a} can be attainable in such way. In particular, the main result of \cite{HK_1} says that if some optimal pair $(\widehat{A},\widehat{y}\,)\in L^2(\Omega;\mathbb{M}^N)\times H^1_0(\Omega)$ to OCP \eqref{2.3}--\eqref{2.3a} is attainable through the above $L^\infty$-approximation of matrix $A^\ast$, then this pair is related by energy equality
\begin{equation}
\label{5.14.c}
\int_\Omega \big(A^{sym}\nabla \widehat{y},\nabla \widehat{y}\big)_{\mathbb{R}^N}\,dx
=\left<f,\widehat{y}\right>_{H^{-1}(\Omega);H^1_0(\Omega)}.
\end{equation}
Hence, the question is what kind of approximation to OCP \eqref{2.3}--\eqref{2.3a} should be applied in order to attain the other types of optimal solutions which do not hold true the energy equality \eqref{5.14.c}.


\section{On approximation of non-variational solutions to OCP (\ref{2.3})--(\ref{2.3a})}
\label{Sec 5}
We begin this section with some auxiliary results and notions. Let $A\in \mathfrak{A}_{ad}$ be a fixed matrix and let
$L(A)$ be a subspace of $H^1_0(\Omega)$ such that
\begin{equation}
\label{4.17aa}
L(A)=\left\{h\in D(A)\ :\ \int_\Omega \big(\nabla \varphi,A\nabla h\big)_{\mathbb{R}^N}\,dx=0\ \forall\,\varphi \in C^\infty_0(\mathbb{R}^N)\right\},
\end{equation}
i.e., $L(A)$ is the set of all weak solutions of the homogeneous problem
\begin{equation}
\label{4.17a}
-\div\big(A\nabla y\big) = 0\quad\text{in }\ \Omega,\quad
y=0\ \text{ on }\ \partial\Omega.
\end{equation}

Let $\e$ be a small parameter.
Assume that the parameter $\e$ varies within a strictly decreasing sequence of positive real numbers which converge to $0$.
Hereinafter in this section,  for any subset $E\subset\Omega$, we denote by
$|E|$ its $N$-dimensional Lebesgue measure $\mathcal{L}^N(E)$.

For every $\e>0$, let $T_\e:\mathbb{R}\rightarrow \mathbb{R}$ be the truncation function defined by
\begin{equation}
\label{5.00}
T_\e(s)=\max\left\{\min\left\{s,\e^{-1}\right\},-\e^{-1}\right\}.
\end{equation}
The following property of $T_\e$ is well known (see \cite{Kind-Sta}). Let $g\in L^2(\Omega)$ be an arbitrary function. Then we have:
\begin{equation}
\label{5.0}
T_\e(g)\in L^\infty(\Omega)\ \forall\,\e>0\quad\text{and}\quad T_\e(g)\rightarrow g\ \text{strongly in }\ L^2(\Omega).
\end{equation}

Let $A^\ast\in L^2\big(\Omega;\mathbb{S}^N_{skew}\big)$ be a matrix mentioned in the control constraints \eqref{2.3c}. For a given sequence $\left\{\e>0\right\}$, we define the cut-off operators $\mathbb{T}_\e:\mathbb{S}^N_{skew}\rightarrow \mathbb{S}^N_{skew}$ as follows
$\mathbb{T}_\e (A^\ast)=\left[T_\e(a^\ast_{ij})\right]_{i,j=1}^N$ for every $\e>0$. We associate with such operators the following
set of subdomains $\left\{\Omega_\e\right\}_{\e>0}$ of $\Omega$
\begin{equation}
\label{5.0a}
\Omega_\e=\Omega\setminus Q_\e,\quad\forall\,\e>0,
\end{equation}
where
\begin{equation}
\label{5.0b}
Q_\e=\mathrm{closure}\,\left\{x\in\Omega\ :\ \|A^\ast(x)\|_{\mathbb{S}^N_{skew}}:=\max_{1\le i<j\le N}\left|a^\ast_{ij}(x)\right|\ge\e^{-1}\right\}.
\end{equation}
\begin{definition}
\label{Def 5.1}
We say that a matrix $A^\ast\in L^2\big(\Omega;\mathbb{S}^N_{skew}\big)$ is of the $\mathfrak{F}$-type, if there exists a strictly decreasing sequence of positive real numbers $\left\{\e\right\}$ converging to $0$ such that the corresponding collection of
sets $\left\{\Omega_\e\right\}_{\e>0}$, defined by \eqref{5.0a}, possesses the following properties:
\begin{enumerate}
\item[(i)] $\Omega_\e$ are  open connected subsets of $\Omega$ with Lipschitz boundaries for which there exists a positive value $\delta>0$ such that
    $$
    \partial\Omega\subset \partial\Omega_\e\quad\text{ and }\quad\mathrm{dist}\, (\Gamma_\e,\partial\Omega)>\delta,\quad \forall\,\e>0,
    $$
    where $\Gamma_\e=\partial\Omega_\e\setminus\partial\Omega$.
\item[(ii)]  The surface measure of the boundaries of holes $Q_\e=\Omega\setminus\Omega_\e$ is small enough in the following sense:
\begin{equation}
\label{5.13c}
\mathcal{H}^{N-1}(\Gamma_\e)= o(\e)\quad \forall\,\e>0.
\end{equation}
\item[(iii)] For each matrix $A\in L^2(\Omega;\mathbb{M}^N)$ such that $A^{skew}\preceq A^\ast$ a.e. in $\Omega$, and for each element
$h\in D(A)$, there is a constant $c=c(h)$ depending on $h$ and independent of $\e$ such that
\begin{equation}
\label{5.0c}
\left|\int_{\Omega\setminus\Omega_\e} \big(\nabla \varphi,A^{skew}\nabla h\big)_{\mathbb{R}^N}\,dx\right|\le c(h)\sqrt{\frac{|\Omega\setminus\Omega_\e|}{\e}} \left(\int_{\Omega\setminus\Omega_\e} |\nabla \varphi|^2_{\mathbb{R}^N}\,dx\right)^{1/2}
\end{equation}
 for all $\varphi\in C^\infty_0(\mathbb{R}^N)$.
\end{enumerate}
\end{definition}

Thus, if $A^\ast$ is of the $\mathfrak{F}$-type, each of the sets $\Omega_\e$ is locally located on one side of its Lipschitz boundary $\partial\Omega_\e$. Moreover, in this case the boundary $\partial\Omega_\e$ can be divided into two parts $\partial\Omega_\e=\partial\Omega\cup\Gamma_\e$. Observe also that  if $A^\ast\in L^\infty\big(\Omega;\mathbb{S}^N_{skew}\big)$ then the estimate \eqref{5.0c} is obviously true for all matrices $A\in L^2(\Omega;\mathbb{M}^N)$ such that $A^{skew}\preceq A^\ast$.

\begin{remark}
\label{Rem 5.1.0}
As immediately follows from Definition~\ref{Def 5.1}, the sequence of perforated domains $\left\{\Omega_\e\right\}_{\e>0}$ is monotonically expanding, i.e., $\Omega_{\e_k} \subset \Omega_{\e_{k+1}}$ for all $\e_k>\e_{k+1}$, and perimeters of $Q_\e$ tend to zero as $\e\to 0$.
Moreover, because of the structure of subdomains $Q_\e$ (see (\ref{5.0b})) and $L^2$-property of the matrix $A^\ast$, we have
$$
\frac{\left|\Omega\setminus\Omega_\e\right|}{\e^2}\le \int_{\Omega\setminus\Omega_\e} \|A^\ast(x)\|^2_{\mathbb{S}^N_{skew}}\,dx,\ \forall\,\e>0\quad\text{and}\quad
\lim_{\e\to 0}\|A^\ast\|_{L^2\left(\Omega\setminus\Omega_\e;\mathbb{S}^N_{skew}\right)}=0.
$$
This entails the property: $\left|\Omega\setminus\Omega_\e\right|=o(\e^2)$ and, hence,
$\lim_{\e\to 0}\left|\Omega_\e\right|=|\Omega|$.  Besides, in view of the condition~(ii) of Definition~\ref{Def 5.1}, we have
\begin{equation}
\label{5.13cc}
\frac{\e\mathcal{H}^{N-1}(\Gamma_\e)}{|\Omega\setminus\Omega_\e|}=O(1).
\end{equation}
\end{remark}

\begin{remark}
\label{Rem 5.1}
As follows from \cite{Cioran}, $\mathfrak{F}$-property of the skew-symmetric matrix $A^\ast$ implies the so-called strong connectedness of the sets $\left\{\Omega_\e\right\}_{\e>0}$ which means the existence of extension operators $P_\e$ from $H^1_0(\Omega_\e;\partial\Omega)$ to $H^1_0(\Omega)$ such that, for some positive constant $C$ independent of $\e$,
\begin{equation}
\label{5.2aa}
\left\|\nabla\left(P_\e y\right)\right\|_{L^2(\Omega;\mathbb{R}^N)}\le C \left\|\nabla  y \right\|_{L^2(\Omega_\e;\mathbb{R}^N)},\quad\forall\, y\in H^1_0(\Omega_\e;\partial\Omega).
\end{equation}
\end{remark}
\begin{remark}
\label{Rem 5.13}
It is easy to see that in view of the conditions (1)--(ii) of Definition~\ref{Def 5.1} and the Sobolev Trace Theorem \cite{Adams}, for all $\e>0$ small enough, the inequality
\begin{equation}
\label{5.13d}
\|\varphi\|_{L^2(\Gamma_\e)}\le \frac{C}{\sqrt{\mathcal{H}^{N-1}(\Gamma_\e)}}\, \|\varphi\|_{H^1_0(\Omega_\e;\partial\Omega)},\quad \forall\varphi\in C^\infty_0(\Omega)
\end{equation}
holds true with a constant $C=C(\Omega)$ independent of $\e$.
\end{remark}

As a direct consequence of Definition \ref{Def 5.1}, we have the following obvious result.
\begin{proposition}
\label{Prop 5.3}
Assume that $A^\ast\in L^2\big(\Omega;\mathbb{S}^N_{skew}\big)$ is of the $\mathfrak{F}$-type.
Let $\left\{\Omega_\e\right\}_{\e>0}$ be a sequence of perforated domains of $\Omega$ given by \eqref{5.0b}, and let
$\left\{\chi_{\Omega_\e}\right\}_{\e>0}$ be the corresponding sequence of characteristic functions. Then
\begin{equation}
\label{5.3a}
\chi_{\Omega_\e}\rightarrow \chi_{\Omega}\quad\text{strongly in }\ L^2(\Omega)\ \text{ and weakly-$\ast$ in }\ L^\infty(\Omega).
\end{equation}
\end{proposition}

\begin{definition}
\label{Def 5.3.1}
We say that a sequence $\left\{y_\e\in H^1_0(\Omega_\e;\partial\Omega)\right\}_{\e>0}$ is weakly convergent in variable spaces $H^1_0(\Omega_\e;\partial\Omega)$ if there exists an element $y\in H^1_0(\Omega)$ such that
\[
\lim_{\e\to 0}\int_{\Omega_\e} \left(\nabla y_\e,\nabla \varphi\right)_{\mathbb{R}^N}\,dx=
\int_{\Omega} \left(\nabla y,\nabla \varphi\right)_{\mathbb{R}^N}\,dx,\quad\forall\,\varphi\in C^\infty_0(\Omega)
\]
\end{definition}
\begin{remark}
\label{Rem 5.3.1}
Let $y^\ast\in H^1_0(\Omega)$ be a weak limit in $H^1_0(\Omega)$  of the extended functions  $\left\{P_\e y_\e\in H^1_0(\Omega)\right\}_{\e>0}$.
Since
\begin{gather*}
\int_{\Omega} \left(\nabla y,\nabla \varphi\right)_{\mathbb{R}^N}\,dx=\lim_{\e\to 0}\int_{\Omega_\e} \left(\nabla y_\e,\nabla \varphi\right)_{\mathbb{R}^N}\,dx =
\lim_{\e\to 0}\int_{\Omega} \left(\nabla \left(P_\e y_\e\right),\nabla \varphi\right)_{\mathbb{R}^N}\chi_{\Omega_\e}\,dx\\\, \stackrel{\text{by \eqref{5.3a} and \eqref{5.2aa}}}{=}\,
\int_{\Omega} \left(\nabla y^\ast,\nabla \varphi\right)_{\mathbb{R}^N}\,dx,\quad\forall\,\varphi\in C^\infty_0(\Omega),
\end{gather*}
it follows that
$\ds\lim_{\e\to 0}\int_{\Omega_\e} \left(\nabla y_\e,\nabla \varphi\right)_{\mathbb{R}^N}\,dx =
\lim_{\e\to 0}\int_{\Omega} \left(\nabla \left(P_\e y_\e\right),\nabla \varphi\right)_{\mathbb{R}^N}\,dx$
and, hence, the weak limit in the sense of Definition \ref{Def 5.3.1} does not depend on the choice of extension operators $P_\e:H^1_0(\Omega_\e;\partial\Omega)\to H^1_0(\Omega)$ with the properties \eqref{5.2aa}.
\end{remark}

Let us consider the following sequence of regularized OCPs associated with perforated domains $\Omega_\e$
\begin{equation}
\label{5.4} \left\{\ \left<\inf_{(A,v,y)\in\Xi_\e}
I_\e(A,v,y)\right>,\quad \e\to 0 \right\},
\end{equation}
where
\begin{gather}
\label{5.5a}
I_\e(A,v,y):=\left\|y-y_d\right\|^2_{L^2(\Omega_\e)}
 + \int_{\Omega_\e}\left(\nabla y, A^{sym}\nabla y\right)_{\mathbb{R}^N}\,dx + \frac{1}{\e^\sigma}\|v\|^2_{H^{-\frac{1}{2}}(\Gamma_\e)},\\[2ex]
\label{5.5b}
\Xi_\e=\left\{(A,v,y)\ \left|\
\begin{array}{c}
-\div\big(A\nabla y\big) = f_\e\quad\text{in }\ \Omega_\e,\\[1ex]
y=0\text{ on }\partial\Omega,\quad
\partial y/ \partial \nu_{A}=v\ \text{on }\Gamma_\e,\\[1ex]
v\in H^{-\frac{1}{2}}(\Gamma_\e),\ y\in H^1_0(\Omega_\e;\partial\Omega),\\[1ex]
A=A^{sym}+A^{skew},\\[1ex]
A\in \mathfrak{A}^\e_{ad}=\mathfrak{A}_{ad,1}\oplus \mathfrak{A}^\e_{ad,2},\
\mathfrak{A}^\e_{ad,2}=U_{a,2}\cap U_{b,2}^\e,\\[1ex]
U_{b,2}^\e=\big\{ A^{skew}=[a_{i\,j}]\in L^2(\Omega;\mathbb{S}^N_{skew})\ :\\[1ex]
\ A^{skew}(x)\preceq A^\ast(x)\ \text{a.e. in }\ \Omega_\e\big\}.
\end{array}
\right.\right\}.
\end{gather}
Here, $y_d\in L^2(\Omega)$ and $f_\e\in
L^2(\Omega)$ are given functions,  $\nu$ is the outward normal unit vector at $\Gamma_\e$ to $\Omega_\e$,
$v\in H^{-\frac{1}{2}}(\Gamma_\e)$ is considered as a fictitious control, and $\sigma$ is a positive number such that
\begin{equation}
\label{5.5c}
\e^{-\sigma}\mathcal{H}^{N-1}(\Gamma_\e)\rightarrow 0\quad\text{as }\ \e\to 0\quad(\text{see \eqref{5.13c}}).
\end{equation}

Using the fact that $A\in L^\infty(\Omega_\e;\mathbb{M}^N)$ for every $\e>0$ and each $A\in \mathfrak{A}^\e_{ad}$, we arrive at the following obvious result.
\begin{theorem}
\label{Th 5.6}
For every $\e>0$ the problem $\left<\inf_{(A,v,y)\in\Xi_\e}
I_\e(A,v,y)\right>$ admits at least one minimizer  $(A^0_\e,v^0_\e,y^0_\e)\in\Xi_\e$.
\end{theorem}

In order to study the asymptotic behavior of the sequences of admissible solutions
$\left\{(A_\e,v_\e,y_\e)\in  \Xi_\e\subset \mathfrak{A}^\e_{ad}\times H^{-\frac{1}{2}}(\Gamma_\e)\times H^1_0(\Omega_\e;\partial\Omega)\right\}_{\e>0}$
in the scale of variable spaces,  we adopt the following concept.
\begin{definition}
\label{Def 5.10} We
say that a sequence
$\left\{(A_\e,v_\e,y_\e)\in  \Xi_\e\right\}_{\e>0}$ weakly converges to a pair $(A,y)\in \mathfrak{A}_{ad}\times H^1_0(\Omega)$ in the scale of spaces
\begin{equation}
\label{5.10a}
\left\{L^2(\Omega;\mathbb{M}^N)\times H^{-\frac{1}{2}}(\Gamma_\e)\times H^1_0(\Omega_\e;\partial\Omega)\right\}_{\e>0},
\end{equation}
(shortly, $(A_\e,v_\e,y_\e)\,\stackrel{w}{\rightarrow}\, (A,y)$), if
\begin{align}
\label{5.10aa}
A_\e:=A_\e^{sym}+A_\e^{skew}&\rightarrow {A}^{sym}+{A}^{skew}=:{A}\ \text{ in }\ L^2(\Omega;\mathbb{M}^N),\\
\label{4.5.1ab}
A_\e^{sym}&\rightarrow {A}^{sym}\quad\text{in }\ L^p(\Omega;\mathbb{S}^N_{sym}),\ \forall\,p\in [1,+\infty),\\
\label{4.5.1ac}
A_\e^{sym}\,&\stackrel{\ast}{\rightharpoonup}\, {A}^{sym}\quad\text{in }\ L^\infty(\Omega;\mathbb{S}^N_{sym}),\\
\label{4.5.1ad} A_\e^{skew}&\rightarrow {A}^{skew}\quad\text{in }\ L^2(\Omega;\mathbb{S}^N_{skew}),\\
\label{5.10b}
 y_\e&\rightharpoonup y\ \text{ in }\ H^1_0(\Omega_\e;\partial\Omega), \\
\label{5.10c}
\text{and}\quad \sup_{\e>0}&\frac{1}{\mathcal{H}^{N-1}(\Gamma_\e)}\,\|v_\e\|^2_{H^{-\frac{1}{2}}(\Gamma_\e)}<+\infty.
\end{align}
\end{definition}

We are now in a position to state the main result of this section.
\begin{theorem}
\label{Th 5.13}
Assume that the matrix $A^\ast\in L^2\big(\Omega;\mathbb{S}^N_{skew}\big)$ is of the $\mathfrak{F}$-type.
Let
$\left\{\Omega_\e\right\}_{\e>0}$ be a sequence of perforated subdomains of $\Omega$ associated with matrix $A^\ast$.
Let $f\in H^{-1}(\Omega)$ and  $y_d\in  L^2(\Omega)$ be given distributions.
Then the original optimal control problem $\left<\inf_{(A,y)\in\Xi} I(A,y)\right>$, where the sequence $\left\{f_\e\in L^2(\Omega)\right\}_{\e>0}$ is such that $\chi_{\Omega_\e}f_\e\rightarrow f$ strongly in $H^{-1}(\Omega)$, is
variational limit of the sequence \eqref{5.4}--\eqref{5.5b} as the parameter $\e$ tends to zero.
\end{theorem}

\begin{proof}
Since each of the optimization problems $\left<\inf_{(A,v,y)\in\Xi_\e}
I_\e(A,v,y)\right>$ lives in the corresponding space $\mathfrak{A}^\e_{ad}\times H^{-\frac{1}{2}}(\Gamma_\e)\times H^1_0(\Omega_\e;\partial\Omega)$, we have to show that in this case  all conditions of Definition~\ref{Def 1.4} hold true. To do so, we divide this proof into two steps.

Step~1. We show on this step that condition (dd) of Definition~\ref{Def 1.4} holds true.  Let $(A,y)\in\Xi$ be an arbitrary admissible pair to the original OCP \eqref{2.3}--\eqref{2.3a}.
We will indicate two cases.
\begin{enumerate}
\item[Case~1.] The set $L(A)$, defined in \eqref{4.17aa}, is a singleton. It means that $h\equiv 0$ is a unique solution of homogeneous problem \eqref{4.17a};
\item[Case~2.] The set $L(A)$ is not a singleton. So, we suppose that the set $L(A)$ is a linear subspace of $H^1_0(\Omega)$ and it contains at least one non-trivial element of $D(A)\subset H^1_0(\Omega)$.
\end{enumerate}

We start with the Case~2. Let $h\in D(A)$ be a element of the set $L(A)$ such that $h$ is a non-trivial solution of homogeneous problem \eqref{4.17a}. In the sequel, the choice of element $h\in L(A)$ will be specified (see \eqref{5.22}). Then we construct a $(\Gamma,0)$-realizing sequence $\left\{(A_\e,v_\e,y_\e)\in\Xi_\e\right\}_{\varepsilon>0}$ in the following way:
\begin{enumerate}
\item[(j)] $A_\e=A$ for all $\e>0$.  In view of definition of the set $\mathfrak{A}_{ad}^\e$, we obviously have that $\left\{A_\e\in \mathfrak{A}_{ad}^\e\subset L^2(\Omega;\mathbb{M}^N)\right\}_{\e>0}$ is a sequence of admissible controls to the problems \eqref{5.4}.
    Note that in this case the properties \eqref{5.10aa}--\eqref{4.5.1ad} are obviously true for the sequence $\left\{A_\e\right\}_{\e>0}$.
 \item[(jj)] Fictitious controls $\left\{v_\e\in H^{-\frac{1}{2}}(\Gamma_\e)\right\}_{\e>0}$ are defined as follows
\begin{equation}
\label{5.13a}
v_\e:= w_\e+\frac{\partial h}{\partial\nu_{A_\e}}\quad\,\forall\, \e>0,
\end{equation}
where distributions $w_\e$ are such that
\begin{equation}
\label{5.20a.2}
\sup_{\e> 0}\left( \frac{1}{\sqrt{\mathcal{H}^{N-1}(\Gamma_\e)}}\, \left\|w_\e\right\|_{H^{-\frac{1}{2}}(\Gamma_\e)}\right)\le C.
\end{equation}
\item[(jjj)] $\left\{y_\e\in H^1_0(\Omega_\e;\partial\Omega)\right\}_{\e>0}$ is the sequence of weak solutions to the corresponding boundary value problems
\begin{gather}
\label{5.13.1aa}
    -\div\big(A\nabla y_\e\big) = -\div\big(A^{sym}\nabla y_\e+A^{skew}\nabla y_\e\big) =f_\e\quad\text{in }\ \Omega_\e,\\
\label{5.13.2aa}
y_\e=0\text{ on }\partial\Omega,\quad
\partial y_\e/ \partial \nu_{A}=v_\e\ \text{on }\Gamma_\e.
\end{gather}
\end{enumerate}
Since $A=\mathbb{T}_\e (A)$ whenever $x\in\Omega_\e$ for every $\e>0$, it means that $A\in L^\infty(\Omega_\e;\mathbb{M}^N)$.
Hence, due to the Lax-Milgram lemma and the superposition principle, the sequence $\left\{y_\e\in H^1_0(\Omega_\e;\partial\Omega)\right\}_{\e>0}$ is defined in a unique way and for every $\e>0$ we have the following decomposition
$y_\e=y_{\e,1}+y_{\e,2}$, where $y_{\e,1}$ and $y_{\e,2}$ are elements of $H^1_0(\Omega_\e)$  such that (hereinafter, we suppose that the functions $y_\e$ of $H^1_0(\Omega_\e,\partial\Omega)$ are extended by operators $P_\e$ outside of $\Omega_\e$)
\begin{alignat}{2}
   \int_{\Omega} \big(\nabla \varphi,A^{sym}\nabla y_{\e,1}&+A^{skew}\nabla y_{\e,1}\big)_{\mathbb{R}^N}\chi_{\Omega_\e}\,dx
=\int_{\Omega}f_\e \chi_{\Omega_\e}\varphi\,dx\notag\\
\label{5.14}
 & +
 \left< w_\e,\varphi\right>_{H^{-\frac{1}{2}}(\Gamma_\e);H^{\frac{1}{2}}(\Gamma_\e)},
 \quad\forall\,\varphi\in C^\infty_0(\Omega;\partial\Omega),\\
\notag
 \int_{\Omega} \big(\nabla \varphi,A^{sym}\nabla y_{\e,2}&+A^{skew}\nabla y_{\e,2}\big)_{\mathbb{R}^N}\chi_{\Omega_\e}\,dx\\
&=\left<\frac{\partial h}{\partial\nu_{A}},\varphi\right>_{H^{-\frac{1}{2}}(\Gamma_\e);H^{\frac{1}{2}}(\Gamma_\e)},\
\forall\,\varphi\in C^\infty_0(\Omega;\partial\Omega).\label{5.15}
\end{alignat}

By the skew-symmetry property of $A^{skew}\in L^\infty(\Omega_\e;\mathbb{S}^N_{skew})$, we have
$$
\int_{\Omega} \big(\nabla y_{\e,i},A^{skew}\nabla y_{\e,i}\big)_{\mathbb{R}^N}\chi_{\Omega_\e}\,dx=0,\quad  i=1,2.
$$
Then \eqref{5.14}--\eqref{5.15} lead us to the energy equalities
\begin{alignat}{2}
   \int_{\Omega} \big(\nabla y_{\e,1}, A^{sym}\nabla y_{\e,1}\big)_{\mathbb{R}^N}\chi_{\Omega_\e}\,dx
&=\int_{\Omega}f_\e \chi_{\Omega_\e}y_{\e,1}\,dx\notag\\
\label{5.16}
 &+
 \left< w_\e,y_{\e,1}\right>_{H^{-\frac{1}{2}}(\Gamma_\e);H^{\frac{1}{2}}(\Gamma_\e)},\\
\label{5.17}
 \int_{\Omega} \big(\nabla y_{\e,2}, A^{sym}\nabla y_{\e,2}\big)_{\mathbb{R}^N}\chi_{\Omega_\e}\,dx
&=\left<\frac{\partial h}{\partial\nu_{A}},y_{\e,2}\right>_{H^{-\frac{1}{2}}(\Gamma_\e);H^{\frac{1}{2}}(\Gamma_\e)}.
\end{alignat}
By the initial assumptions, we have $h\in L(A)$. Then the condition (iii) of Definition~\ref{Def 5.1} implies that
(for the details we refer to \cite{KPI})
\begin{multline*}
\left|\left<\frac{\partial h}{\partial\nu_{A}},\varphi\right>_{H^{-\frac{1}{2}}(\Gamma_\e);H^{\frac{1}{2}}(\Gamma_\e)}\right|=
\left|\int_{\Omega\setminus\Omega_\e} \big(\nabla \varphi,A^{sym}\nabla h+A^{skew}\nabla h\big)_{\mathbb{R}^N}\,dx\right|\\
\le \sqrt{\frac{|\Omega\setminus\Omega_\e|}{\e}}\left(C_1(h)+C_2(h)\right)\|\varphi\|_{H^1(\Omega\setminus\Omega_\e)}\\
\stackrel{\text{by \eqref{5.13c}}}{\le}\,
C(h)\sqrt{\mathcal{H}^{N-1}(\Gamma_\e)}\|\varphi\|_{H^1(\Omega\setminus\Omega_\e)},\quad\forall\,\varphi \in H^1_0(\Omega)
\end{multline*}
with some constant $C(h)$ independent of $\e$.
Hence,
\begin{equation}
\label{5.20a.1}
\sup_{\e>0}\left(\mathcal{H}^{N-1}(\Gamma_\e)\right)^{-1}\,\Big\|\frac{\partial h}{\partial\nu_{A}}\Big\|^2_{H^{-\frac{1}{2}}(\Gamma_\e)}<C(h)<+\infty.
\end{equation}
Thus, using the continuity of the embedding $H^{\frac{1}{2}}(\Gamma_\e)\hookrightarrow L^2(\Gamma_\e)$ and Sobolev Trace Theorem, we get
\begin{alignat}{2}
\notag
\Big|\left< w_\e,y_{\e,1}\right>_{H^{-\frac{1}{2}}(\Gamma_\e);H^{\frac{1}{2}}(\Gamma_\e)}\Big|&\stackrel{\text{by \eqref{5.20a.2}}}{\le}\, C\,\|y_{\e,1}\|_{ L^2(\Gamma_\e)}\left(\mathcal{H}^{N-1}(\Gamma_\e)\right)^{\frac{1}{2}}\\
&\stackrel{\text{by \eqref{5.13d}}}{\le}\, C_1\,\|y_{\e,1}\|_{H^1_0(\Omega_\e;\partial\Omega)},\label{5.20a}\\
\notag
\Big|\left<\frac{\partial h}{\partial\nu_{A}},y_{\e,2}\right>_{H^{-\frac{1}{2}}(\Gamma_\e);H^{\frac{1}{2}}(\Gamma_\e)}\Big|&\le C\,\|y_{\e,2}\|_{ L^2(\Gamma_\e)}\left(\mathcal{H}^{N-1}(\Gamma_\e)\right)^{\frac{1}{2}}\\
&\stackrel{\text{by \eqref{5.13d}}}{\le}\, C_1 \|y_{\e,2}\|_{H^1_0(\Omega_\e;\partial\Omega)}\label{5.20b}.
\end{alignat}
As a result, we arrive at the following the a priori estimates
\begin{align}
\label{5.18}
   \left(\int_{\Omega} \big\|\nabla y_{\e,1}\big\|^2_{\mathbb{R}^N}\chi_{\Omega_\e}\,dx\right)^{1/2}  &\le \alpha^{-1}\left(\|f_\e\chi_{\Omega_\e}\|_{H^{-1}(\Omega)}
   +C\right),\\
\label{5.19}
\left(\int_{\Omega} \big\|\nabla y_{\e,2}\big\|^2_{\mathbb{R}^N}\chi_{\Omega_\e}\,dx\right)^{1/2} &\le C\alpha^{-1}.
\end{align}
Hence, the sequences
$\left\{y_{\e,1}\in H^1_0(\Omega_\e;\partial\Omega)\right\}_{\e>0}$ and $\left\{y_{\e,2}\in H^1_0(\Omega_\e;\partial\Omega)\right\}_{\e>0}$
are weakly compact with respect to the weak convergence in variable spaces \cite{Zh_98}, i.e., we may assume that there exists a couple of functions $\widehat{y}_1$ and $\widehat{y}_2$ in $H^1_0(\Omega)$ such that
\begin{equation}
\label{5.20}
\lim_{\e\to 0}\int_{\Omega} \big(\nabla \varphi,\nabla y_{\e,i}\big)_{\mathbb{R}^N}\chi_{\Omega_\e}\,dx=
\int_{\Omega} \big(\nabla \varphi,\nabla \widehat{y}_{i}\big)_{\mathbb{R}^N}\\,dx,\
\forall\,\varphi\in C^\infty_0(\Omega), \ i=1,2.
\end{equation}

Now we can pass to the limit in the integral identities \eqref{5.14}--\eqref{5.15} as $\e\to 0$. Using  \eqref{5.20a.2}, \eqref{5.20},  \eqref{5.20a.1}, $L^2$-property of $A\in \mathfrak{A}_{ad}$, and the fact that $\chi_{\Omega_\e}f_\e\rightarrow f$ strongly in $H^{-1}(\Omega)$, we finally obtain
\begin{gather}
\label{5.21a}
\int_{\Omega} \big(\nabla \varphi,A^{sym}\nabla \widehat{y}_{1}+A^{skew}\nabla \widehat{y}_{1}\big)_{\mathbb{R}^N}\,dx
=\left<f, \varphi\right>_{H^{-1}(\Omega);H^1_0(\Omega)}\\
\label{5.21b}
\int_{\Omega} \big(\nabla \varphi,A^{sym}\nabla \widehat{y}_{2}+A^{skew}\nabla \widehat{y}_{2}\big)_{\mathbb{R}^N}\,dx=0
\end{gather}
for every $\varphi\in C^\infty_0(\Omega)$. Hence, $\widehat{y}_1$ and $\widehat{y}_2$ are weak solutions to the boundary value problem \eqref{2.1}--\eqref{2.2} and \eqref{4.17a}, respectively. Hence,  $\widehat{y}_2\in L(A)$ and  $\widehat{y}_1\in D(A)$ (see \cite{HK_1}).  As a result, we arrive at the conclusion: the pair $(A,\widehat{y}_1+h)$ belongs to the set $\Xi$, for every $h\in L(A)$.
Since  by the initial assumptions $(A,y)\in\Xi$, it follows that having set in \eqref{5.13a}
\begin{equation}
\label{5.22}
h=y-\widehat{y}_1,
\end{equation}
we obtain
\begin{equation}
\label{5.23}
h\in L(A)\ \text{ and }\ y_\e=y_{\e,1}+y_{\e,2}\rightharpoonup y\quad\text{ in }\ H^1_0(\Omega_\e;\partial\Omega)\ \text{ as }\ \e\to 0.
\end{equation}
Therefore, in view of \eqref{5.23}, \eqref{5.20a.1}, \eqref{5.20a.2}, we see that
$$
(A_\e,v_\e,y_\e)\,\stackrel{w}{\rightarrow}\, (A,y)\quad\text{in the sense of Definition~\ref{Def 5.10}}.
$$
Thus,
the property \eqref{1.7a} holds true.
It is worth to notice that in the Case~1, we can give the same conclusion, because we originally have $h\equiv 0$. Hence, the solutions to boundary value problems \eqref{5.21a}--\eqref{5.21a} are unique and, therefore, we can claim that $y=\widehat{y}_1$, $\widehat{y}_2=0$, and $h=0$.

It remains to prove the inequality \eqref{1.7c}. To do so, it is enough to show that
\begin{align}
\notag
I&(A,y):= \left\|y-y_d\right\|^2_{L^2(\Omega)}
 + \int_\Omega\left(\nabla y, A^{sym}\nabla y\right)_{\mathbb{R}^N}\,dx
 = \lim_{\e\to 0}I_\e(u_\e,v_\e,y_\e)\\
 = &
 \lim_{\e\to 0}\Big[\left\|y_\e-y_d\right\|^2_{L^2(\Omega_\e)}
  + \int_{\Omega_\e}\left(\nabla y_\e, A^{sym}\nabla y_\e\right)_{\mathbb{R}^N}\,dx + \frac{1}{\e^\sigma}\|v_\e\|^2_{H^{-\frac{1}{2}}(\Gamma_\e)}\Big],\label{5.23a}
\end{align}
where the sequence $\left\{(u_\e,v_\e,y_\e)\in\Xi_\e\right\}_{\varepsilon>0}$ is defined by \eqref{5.13a} and \eqref{5.22}.

In view of this, we make use the following relations
\begin{gather}
\label{5.24}
\left.
\begin{array}{c}
\ds \|v_\e\|^2_{H^{-\frac{1}{2}}(\Gamma_\e)}\le 2\|w_\e\|^2_{H^{-\frac{1}{2}}(\Gamma_\e)}+2\Big\|\frac{\partial h}{\partial\nu_{A}}\Big\|^2_{H^{-\frac{1}{2}}(\Gamma_\e)}<+\infty,\\[2ex]
\ds\lim_{\e\to 0} \frac{1}{\e^\sigma}\|w_\e\|^2_{H^{-\frac{1}{2}}(\Gamma_\e)}\,\stackrel{\text{by \eqref{5.20a.2}}}{\le}\, C
\lim_{\e\to 0} \frac{\mathcal{H}^{N-1}(\Gamma_\e)}{\e^\sigma}=0,\\[2ex]
\ds\lim_{\e\to 0} \frac{1}{\e^\sigma}\left\|\frac{\partial h}{\partial\nu_{A}}\right\|^2_{H^{-\frac{1}{2}}(\Gamma_\e)}\,\stackrel{\text{by \eqref{5.20a.1}}}{\le}\, C
\lim_{\e\to 0} \frac{\mathcal{H}^{N-1}(\Gamma_\e)}{\e^\sigma}=0,\\[1ex]
\ds\lim_{\e\to 0}\left\|y_\e-y_d\right\|^2_{L^2(\Omega_\e)}\,\stackrel{\text{by \eqref{5.3a} and \eqref{5.23}}}{=}\,
\left\|y-y_d\right\|^2_{L^2(\Omega)}.
\end{array}
\right\}
\end{gather}

In order to obtain the convergence
\begin{equation}\label{5.25}
\lim\limits_{\e \to 0}
\int_{\Omega_\e}\left(\nabla y_\e, A^{sym}\nabla y_\e\right)_{\mathbb{R}^N}\,dx =
\int_\Omega\left(\nabla y, A^{sym}\nabla y\right)_{\mathbb{R}^N}\,dx,
\end{equation}
we apply the energy equality
which comes from the condition $(A,y)\in\Xi$
\begin{equation}
\int_\Omega\left(\nabla y, A^{sym}\nabla y\right)_{\mathbb{R}^N}\,dx
=-[y,y]_{A}+\left<f, y\right>_{H^{-1}(\Omega);H^1_0(\Omega)},\label{5.27}
\end{equation}
and make use of the following trick.
It is easy to see that
the integral identity for the weak solutions $y_\e$ to boundary value problems \eqref{5.5b} can be represented in the so-called extended form
\begin{alignat}{2}
   \int_{\Omega} \big(\nabla \varphi,A^{sym}\nabla y_{\e}&+A^{skew}\nabla y_{\e}\big)_{\mathbb{R}^N}\chi_{\Omega_\e}\,dx
=\int_{\Omega}f_\e \chi_{\Omega_\e}\varphi\,dx\notag\\
 & +
 \left< w_\e,\varphi\right>_{H^{-\frac{1}{2}}(\Gamma_\e);H^{\frac{1}{2}}(\Gamma_\e)} + \left<\frac{\partial h}{\partial\nu_{A}},\varphi\right>_{H^{-\frac{1}{2}}(\Gamma_\e);H^{\frac{1}{2}}(\Gamma_\e)}\notag\\
 \label{5.26.1}
 &- \int_{\Omega} \big(\nabla \psi,A^{sym}\nabla h^\ast\big)_{\mathbb{R}^N}\,dx-[h^\ast,\psi]_A,\quad
\forall\,\varphi, \psi\in C^\infty_0(\Omega),
\end{alignat}
where $h^\ast$ is an arbitrary element of $L$. Indeed, because of the equality
\[
\int_{\Omega} \big(\nabla \psi,A^{sym}\nabla h^\ast\big)_{\mathbb{R}^N}\,dx+[h^\ast,\psi]_A\,\stackrel{\text{by \eqref{4.17aa}}}{=}\,0,\quad\forall\,\psi\in C^\infty_0(\Omega),
\]
we have an equivalent identity to the classical definition of the weak solutions of boundary value problem \eqref{5.5b}.

As follows from \eqref{5.20a.1}, \eqref{5.23}, and the Sobolev Trace Theorem, the numerical sequences
$$
\left\{\left< w_\e,y_\e\right>_{H^{-\frac{1}{2}}(\Gamma_\e);H^{\frac{1}{2}}(\Gamma_\e)}\right\}_{\e>0}\quad\text{and}\quad
\left\{ \left<\frac{\partial h}{\partial\nu_{A}},y_\e\right>_{H^{-\frac{1}{2}}(\Gamma_\e);H^{\frac{1}{2}}(\Gamma_\e)}\right\}_{\e>0}
$$
are bounded.
Therefore, we can assume, passing to a subsequence if necessary, that there exists a value $\xi_1\in \mathbb{R}$ such that
\begin{equation}
\label{5.24b}
\left< w_\e,y_\e\right>_{H^{-\frac{1}{2}}(\Gamma_\e);H^{\frac{1}{2}}(\Gamma_\e)}+
\left<\frac{\partial h}{\partial\nu_{A}},y_\e\right>_{H^{-\frac{1}{2}}(\Gamma_\e);H^{\frac{1}{2}}(\Gamma_\e)}
\longrightarrow \xi_1\quad\text{as }\ \e\to 0.
\end{equation}
Since $y_\e\rightharpoonup y$ weakly in $H^1_0(\Omega_\e;\partial\Omega)$ and $y\in D(A)$, it follows that there exists a sequence of smooth functions $\left\{\psi_\e\in C^\infty_0(\Omega)\right\}_{\e>0}$ such that $\psi_\e\rightarrow y$ strongly in $H^1_0(\Omega)$. Therefore, following the extension rule \eqref{2.8a}, we have
\begin{align}
\label{5.26.2.1}
\lim_{\e\to 0}\int_{\Omega} \big(\nabla \psi_\e,A^{sym}\nabla h^\ast\big)_{\mathbb{R}^N}\,dx&=\int_\Omega \big(\nabla y,A^{sym}\nabla h^\ast\big)_{\mathbb{R}^N}\,dx,\\
\label{5.26.2.2}
\lim_{\e\to 0} [h^\ast,\psi_\e]_A&=[h^\ast,y]_A.
\end{align}

Because of the initial assumptions, we can assume that the element $h^\ast\in L(A)$ is such that
$$[h^\ast,y]_A+\int_\Omega \big(\nabla y,A^{sym}\nabla h^\ast\big)_{\mathbb{R}^N}\,dx\ne 0.$$
So, due to this observation, we specify the choice of element $h^\ast\in L(A)$ as follows
\[
\widehat{h}^\ast=\frac{\xi_1+[y,y]_A}{\xi_2+\xi_3}\, h^\ast,\quad\text{where }\ \xi_3:=\int_\Omega \big(\nabla y,A^{sym}\nabla h^\ast\big)_{\mathbb{R}^N}\,dx,\ \xi_2:=[h^\ast,y]_A,
\]
or, in other words, we aim to ensure the condition
$\xi_1-\xi_2-\xi_3 +[y,y]_A=0$.
As a result, we have:  $\widehat{h}^\ast$ is an element of $L(A)$ such that
\begin{equation}
\label{5.26.3}
\lim_{\e\to 0}\int_{\Omega} \big(\nabla \psi_\e,\nabla \widehat{h}^\ast\big)_{\mathbb{R}^N}\,dx=\xi_2\frac{\xi_1+[y,y]_A}{\xi_2+\xi_3},\quad
\lim_{\e\to 0} [\widehat{h}^\ast,\psi_\e]=\xi_3\frac{\xi_1+[y,y]_A}{\xi_2+\xi_3}.
\end{equation}

Having put $\varphi=y_\e$ and $h^\ast=\widehat{h}^\ast$ in \eqref{5.26.1} and using the fact that
$$
\int_{\Omega} \big(\nabla y_\e,A^{skew} \nabla y_{\e}\big)_{\mathbb{R}^N}\chi_{\Omega_\e}\,dx=0,
$$
we arrive at the following energy equality for the boundary value problem \eqref{5.5b}
\begin{gather}
   \int_{\Omega} \big(\nabla y_\e,A^{sym}\nabla y_{\e}\big)_{\mathbb{R}^N}\chi_{\Omega_\e}\,dx
=\int_{\Omega}f_\e \chi_{\Omega_\e}y_\e\,dx
  +
 \left< w_\e,y_\e\right>_{H^{-\frac{1}{2}}(\Gamma_\e);H^{\frac{1}{2}}(\Gamma_\e)}\notag\\
 + \left<\frac{\partial h}{\partial\nu_{A}},y_\e\right>_{H^{-\frac{1}{2}}(\Gamma_\e);H^{\frac{1}{2}}(\Gamma_\e)}
 \label{5.26.4}
 - \int_{\Omega} \big(\nabla \psi_\e,A^{sym}\nabla \widehat{h}^\ast\big)_{\mathbb{R}^N}\,dx-[\widehat{h}^\ast,\psi_\e]_A.
\end{gather}

As a result, taking into account the properties \eqref{5.3a}, \eqref{5.23}, \eqref{5.26.3}, we can pass to the limit as $\e\to 0$ in \eqref{5.26.4}. This yields
\begin{align}
\notag
\lim_{\e\to 0}\int_{\Omega}& \big(\nabla y_\e,A^{sym}\nabla y_{\e}\big)_{\mathbb{R}^N}\chi_{\Omega_\e}\,dx=\lim_{\e\to 0}\int_{\Omega}f_\e \chi_{\Omega_\e}y_{\e}\,dx\\\notag&+
\lim_{\e\to 0}\left< w_\e,y_{\e}\right>_{H^{-\frac{1}{2}}(\Gamma_\e);H^{\frac{1}{2}}(\Gamma_\e)}+ \lim_{\e\to 0}\left< \frac{\partial h}{\partial\nu_{A}},y_{\e}\right>_{H^{-\frac{1}{2}}(\Gamma_\e);H^{\frac{1}{2}}(\Gamma_\e)}\\\notag&
-\lim_{\e\to 0}\int_{\Omega} \big(\nabla \psi_\e,\nabla \widehat{h}^\ast\big)_{\mathbb{R}^N}\,dx -
\lim_{\e\to 0} [\widehat{h}^\ast,\psi_\e]_A\\
\,\stackrel{\text{by \eqref{5.26.3}}}{=}&\left<f,y\right>_{H^{-1}(\Omega);H^1_0(\Omega)}-[y,y]_A\,\stackrel{\text{by \eqref{5.27}}}{=}\, \int_\Omega\left(\nabla y, A^{sym}\nabla y\right)_{\mathbb{R}^N}\,dx.\label{5.28}
\end{align}
Hence, turning back to \eqref{5.23a}, we see that this relation is a direct consequence of \eqref{5.24} and \eqref{5.28}. Thus,
the sequence $\left\{(u_\e,v_\e,y_\e)\in\Xi_\e\right\}_{\varepsilon>0}$, which is defined by \eqref{5.13a} and \eqref{5.22}, is $\Gamma$-realizing. The property (dd)  is established.

Step~2. We prove the property (d) of Definition~\ref{Def 1.4}. Let
$\left\{(A_k,v_k,y_k)\right\}_{k\in \mathbb{N}}$ be a sequence such that  $(A_k,v_k,y_k)\in
\Xi_{\varepsilon_k}$ for some $\e_k\to 0$ as $k\to\infty$,
\begin{equation}
\label{5.28.28}
\left.
\begin{split}
A_k:=A_k^{sym}+A_k^{skew}&\rightarrow {A}^{sym}+{A}^{skew}=:{A}\ \text{ in }\ L^2(\Omega;\mathbb{M}^N),\\
A_k^{sym}&\rightarrow {A}^{sym}\quad\text{in }\ L^p(\Omega;\mathbb{S}^N_{sym}),\ \forall\,p\in [1,+\infty),\\
A_k^{sym}\,&\stackrel{\ast}{\rightharpoonup}\, {A}^{sym}\quad\text{in }\ L^\infty(\Omega;\mathbb{S}^N_{sym}),\\
A_\e^{skew}&\rightarrow {A}^{skew}\quad\text{in }\ L^2(\Omega;\mathbb{S}^N_{skew}),\\
y_k&\rightharpoonup y\ \text{ in }\ H^1_0(\Omega_{\e_k};\partial\Omega),
\end{split}
\right\}
\end{equation}
and the sequence of fictitious controls $\left\{v_k\in H^{-\frac{1}{2}}(\Gamma_{\e_k})\right\}_{k\in \mathbb{N}}$ satisfies inequality \eqref{5.10c}. In view of Definition~\ref{Def 5.10} it means that
$(A_k,v_k,y_k)\,\stackrel{w}{\rightarrow}\, (A,y)$ as $k\to\infty$.
Our aim is to show that
\begin{equation}
\label{5.29} (A,y)\in \Xi \quad\text{and}\quad I(A,y)\le
\liminf_{k\to\infty}I_{\e_k}(A_k,v_k,y_k).
\end{equation}
It is easy to see that the limit matrix $A$ is an admissible control to OCP \eqref{2.3}--\eqref{2.3a}, i.e. $A\in \mathfrak{A}_{ad}$. Since the integral identity
\begin{multline}
\label{5.30}
\int_{\Omega} \big(\nabla \varphi,A_k^{sym}\nabla y_{k}+A_k^{skew}\nabla y_{k}\big)_{\mathbb{R}^N}\chi_{\Omega_{\e_k}}\,dx
=\int_{\Omega}f_{\e_k} \chi_{\Omega_{\e_k}}\varphi\,dx\\+
 \left< v_k,\varphi\right>_{H^{-\frac{1}{2}}(\Gamma_{\e_k});H^{\frac{1}{2}}(\Gamma_{\e_k})},\quad
 \forall\,\varphi\in C^\infty_0(\Omega)
\end{multline}
holds true for every $k\in \mathbb{N}$, we can pass to the limit in \eqref{5.30} as $k\to\infty$ using Definition \ref{Def 5.10}  and the estimate
\[
\Big|\left< v_k,\varphi\right>_{H^{-\frac{1}{2}}(\Gamma_{\e_k});H^{\frac{1}{2}}(\Gamma_{\e_k})}\Big|\le C(\Omega)\,\|\varphi\|_{H^1_0(\Omega)}\left(\mathcal{H}^{N-1}(\Gamma_{\e_k})\right)^{\frac{1}{2}},\quad\forall\,\varphi\in C^\infty_0(\Omega)
\]
coming from inequality \eqref{5.10c}. Then proceeding as on the Step~1, it can easily be shown that the limit pair $(A,y)$ is admissible to OCP \eqref{2.3}--\eqref{2.3a}. Hence, the condition \eqref{5.29}$_1$ is valid.

As for the inequality \eqref{5.29}$_2$, we see that
\begin{equation}
\label{5.30.1}
\lim_{k\to\infty} \left\|y_k-y_d\right\|^2_{L^2(\Omega_{\e_k})}=\lim_{k\to\infty}\left\|(y_k-y_d)\chi_{\Omega_{\e_k}}\right\|^2_{L^2(\Omega)}
=\left\|y-y_d\right\|^2_{L^2(\Omega)}
\end{equation}
by \eqref{5.3a} and compactness of the embedding $H^1_0(\Omega)\hookrightarrow L^2(\Omega)$. In view of the properties \eqref{5.28.28} and \eqref{1.1}, the sequence $\left\{\left(A_k^{sym}\right)^{1/2}\right\}_{k\in \mathbb{N}}$ is obviously bounded in $L^2(\Omega; \mathbb{S}^N_{sym})$. Moreover, taking into account the norm convergence property
\begin{gather*}
\lim_{k\to\infty}\|\left(A_k^{sym}\right)^{1/2}\xi\|^2_{L^2(\Omega;\mathbb{R}^N)}=
\lim_{k\to\infty}\int_{\Omega}\left(\xi, A_k^{sym} \xi\right)_{\mathbb{R}^N}\,dx\\=
\int_\Omega\left(\xi, A^{sym} \xi\right)_{\mathbb{R}^N}\,dx=
\|\left(A^{sym}\right)^{1/2}\xi\|^2_{L^2(\Omega;\mathbb{R}^N)},\quad\forall\,\xi\in \mathbb{R}^N,
\end{gather*}
we can conclude that the sequence $\left\{\left(A_k^{sym}\right)^{1/2}\right\}_{k\in \mathbb{N}}$ strongly converges to $\left(A^{sym}\right)^{1/2}$ in $L^2(\Omega;\mathbb{S}^N_{sym})$. Hence, combining this fact with \eqref{5.28.28}$_5$ and \eqref{5.3a}, we finally obtain
\[
\chi_{\Omega_{\e_k}}\left(A_k^{sym}\right)^{1/2}\nabla y_k\rightharpoonup \chi_{\Omega}\left(A^{sym}\right)^{1/2}\nabla y\quad\text{in }\ L^2(\Omega;\mathbb{R}^N).
\]
As a result, the lower semicontinuity of $L^2$-norm with respect to the weak convergence, immediately leads us to the inequality
\begin{align}
\notag
\liminf_{k\to\infty} &\int_{\Omega_{\e_k}}\left(\nabla y_k, A_k^{sym} \nabla y_k\right)_{\mathbb{R}^N}\,dx=
\liminf_{k\to\infty} \|\chi_{\Omega_{\e_k}}\left(A_k^{sym}\right)^{1/2}\nabla y_k\|^2_{L^2(\Omega;\mathbb{R}^N)}\\
&\ge \|\chi_{\Omega}\left(A^{sym}\right)^{1/2}\nabla y\|^2_{L^2(\Omega;\mathbb{R}^N)}=
\int_{\Omega}\left(\nabla y, A^{sym} \nabla y\right)_{\mathbb{R}^N}\,dx.
\label{5.30.2}
\end{align}

Thus, in order to prove the inequality \eqref{5.29}$_2$, it remains to combine relations \eqref{5.30.1}, \eqref{5.30.2}, and take into account the following estimate
\begin{equation}
\label{5.30.3}
\frac{1}{{(\e_k)}^\sigma}\,\|v_k\|^2_{H^{-\frac{1}{2}}(\Gamma_{\e_k})}
\le C
 \frac{\mathcal{H}^{N-1}(\Gamma_{\e_k})}{(\e_k)^\sigma}\rightarrow 0\quad\text{as }\ k\to\infty.
\end{equation}
The proof is complete.
\end{proof}

In conclusion of this section, we consider the variational properties of OCPs \eqref{5.4}--\eqref{5.5b}. To this end, we apply
Theorem~\ref{Th 1.8}.
\begin{theorem}
\label{Th 5.31}
Let $A^\ast\in L^2\big(\Omega;\mathbb{S}^N_{skew}\big)$ be a matrix of the $\mathfrak{F}$-type.
Let $y_d\in L^2(\Omega)$ and $f\in
H^{-1}(\Omega)$ be given distributions.
Let $\left\{(A^0_\varepsilon,v^0_\e,y^0_\e)\in \Xi_\varepsilon\right\}_{\e>0}$ be a sequence of optimal solutions to regularized problems \eqref{5.4}--\eqref{5.5b}, where $\chi_{\Omega_\e}f_\e\rightarrow f$ strongly in $H^{-1}(\Omega)$.
Then there exists an optimal pair $(A^0,y^0)\in \mathfrak{A}_{ad}$ to the original OCP \eqref{2.3}--\eqref{2.3a}, which is attainable in the following sense
\begin{gather}
\label{5.32}
(A^0_\varepsilon,v^0_\e,y^0_\e)\,\stackrel{w}{\rightarrow}\, (A^0,y^0)\ \text{ as }\ \e\to 0\\\notag \text{ in variable space }\ L^2(\Omega;\mathbb{M}^N)\times H^{-\frac{1}{2}}(\Gamma_\e)\times H^1_0(\Omega_\e;\partial\Omega),\\
\inf_{(A,y)\in\,\Xi}I(A,y)= I\left(A^0,y^0\right) =\lim_{\e\to
0} I_{\e}(A^0_\e,v^0_\e,y^0_\e)=\lim_{\e\to
0} \inf_{(A,v,y)\in\Xi_\e}
I_\e(A,v,y).\label{5.33}
\end{gather}
\end{theorem}
\begin{proof}
In order to show that this result is a direct consequence of Theorem \ref{Th 1.8}, it is enough to establish the compactness property for the sequence of optimal solutions $\left\{(A^0_\e,v^0_\e,y^0_\e)\in \Xi_\varepsilon\right\}_{\e>0}$ in  the sense of Definition \ref{Def 5.10}.

Let $h\in C^\infty_0(\Omega)$ be a non-zero function such that
$\mathrm{div}\,\left(A^{sym}\nabla h + A^\ast\nabla h\right)\in L^2(\Omega)$, where we assume that $A=A^{sym}+A^\ast$ is an admissible control, $A\in \mathfrak{A}_{ad}$. We set $v_\e=\left.\frac{\partial h}{\partial\nu_A}\right|_{\Gamma_\e}\in H^{-\frac{1}{2}}(\Gamma_\e)$. In view of the initial assumptions and estimate (see \cite{KPI} for the details)
\begin{equation*}
\sup_{\e> 0}\left( \frac{1}{\sqrt{\mathcal{H}^{N-1}(\Gamma_\e)}}\, \Big\|\frac{\partial h}{\partial\nu_{A}}\Big\|_{H^{-\frac{1}{2}}(\Gamma_\e)}\right)\le C.
\end{equation*}
there is a constant $C>0$ independent of $\e$ such that
\[
 \Big\|\frac{\partial h}{\partial\nu_{A}}\Big\|^2_{H^{-\frac{1}{2}}(\Gamma_\e)}\le C {\mathcal{H}^{N-1}(\Gamma_\e)},
\]

Let $y_\e=y_\e(A_\e,v_\e,f)\in H^1_0(\Omega_\e;\partial\Omega)$ be a corresponding solution to boundary value problem \eqref{5.5b}.
Then following \eqref{5.18}, we come to the estimate
$$\|y_\e\|^2_{H^1_0(\Omega_\e;\partial\Omega)}\le \widetilde{C},$$
where the constant $\widetilde{C}$ is also independent of $\e$. As a result, we get
\begin{align*}
I_{\e}&(A^0_\e,v^0_\e,y^0_\e)=\left\|y^0_\e-y_d\right\|^2_{L^2(\Omega_\e)}
 + \int_{\Omega_{\e}}\left(\nabla y_\e^0, (A_\e^0)^{sym} \nabla y_\e^0\right)_{\mathbb{R}^N}\,dx\\& + \frac{1}{{\e}^\sigma}\|v^0_\e\|^2_{H^{-\frac{1}{2}}(\Gamma_\e)}
\le I_\e(A_\e,v_\e,y_\e)\le (2C_1+\beta)\widetilde{C} +2\|y_d\|^2_{L^2(\Omega)} + C\frac{\mathcal{H}^{N-1}(\Gamma_{\e})}{{\e}^\sigma}.
\end{align*}
Since ${\e}^{-\sigma}\mathcal{H}^{N-1}(\Gamma_{\e})\rightarrow 0$ as $\e\to 0$, it follows that
the minimal values of the cost functional \eqref{5.5a} bounded above
uniformly with respect to $\e$.
Thus, the sequence of optimal solutions $\left\{(A^0_\e, v^0_\e,y^0_\e)\right\}_{\e>0}$ to the problems \eqref{5.4}--\eqref{5.5b} uniformly bounded in $L^2(\Omega;\mathbb{M}^N)\times H^{-\frac{1}{2}}(\Gamma_\e)\times H^1_0(\Omega_\e)$ and, hence, in view of Proposition~\ref{Prop 2.3f} , it is relatively compact with respect to the weak convergence in the sense of Definition \ref{Def 5.10}. For the rest of proof, it remains to apply Theorem~\ref{Th 1.8}.
\end{proof}

\begin{remark}
We note that variational properties of optimal solutions, given by Theorem~\ref{Th 5.31}, do not suffice to assert that the convergence of optimal states $P_\e(y_\e^0)$ to $y^0$ is strong in $H^1_0(\Omega)$. Indeed, the convergence
\begin{equation}
\label{5.35}
\int_{\Omega_\e} \big(\nabla y^0_\e,\left(A_\e^0\right)^{sym}\nabla y^0_{\e}\big)_{\mathbb{R}^N}\,dx \,\stackrel{\e\to 0}{\longrightarrow}\,
\int_{\Omega_\e} \big(\nabla y^0,\left(A^0\right)^{sym}\nabla y^0\big)_{\mathbb{R}^N}\,dx,
\end{equation}
which comes from \eqref{5.32}--\eqref{5.33}, does not imply the norm convergence in $H^1_0(\Omega)$. At the same time, combining relation \eqref{5.35} with energy identities
\[
 \int_{\Omega_\e} \big(\nabla y^0_\e,\left(A_\e^0\right)^{sym}\nabla y^0_{\e}\big)_{\mathbb{R}^N}\,dx
=\int_{\Omega_\e}f_\e y^0_\e\,dx
  +
 \left< v^0_\e,y^0_\e\right>_{H^{-\frac{1}{2}}(\Gamma_\e);H^{\frac{1}{2}}(\Gamma_\e)}
\]
and
\[
\int_\Omega\left(\nabla y^0, \left(A^0\right)^{sym}\nabla y^0\right)_{\mathbb{R}^N}\,dx
=-[y^0,y^0]_{A^0}+\left<f, y^0\right>_{H^{-1}(\Omega);H^1_0(\Omega)}
\]
rewritten for optimal solutions of the problems \eqref{5.13.1aa}--\eqref{5.13.2aa} and \eqref{2.1}--\eqref{2.2}, respectively, we get
\begin{equation}
\label{5.34}
\lim_{\e\to 0} \left<v^0_\e,y^0_{\e}\right>_{H^{-\frac{1}{2}}(\Gamma_\e);H^{\frac{1}{2}}(\Gamma_\e)}=-[y^0,y^0]_{A^0}.
\end{equation}
It gives us another example of the product of two weakly convergent sequences that can be recovered in the limit in an explicit form. Moreover, this limit does not coincide with the product of their weak limits.
\end{remark}

Our next remark deals with a motivation to put forward another concept of the weak solutions to the approximated boundary value problem \eqref{5.5b} which can be viewed as a refinement of the integral identity \eqref{5.14}.
\begin{definition}
\label{Def 5.30.0a}
Let
$\left\{\Omega_\e\right\}_{\e>0}$ be a sequence of perforated subdomains of $\Omega$ associated with matrix $A$ by the rule \eqref{5.0a}--\eqref{5.0b}.
We say that a function $y_\e=y_\e(A,f,v)\in H^1_0(\Omega_\e)$ is  a weak
solution to the boundary value problem \eqref{5.5b} for
given $A\in \mathfrak{A}_{ad}$,  $f_\e\in
L^2(\Omega)$, and $v\in H^{-\frac{1}{2}}(\Gamma_\e)$, if
 the relation
\begin{multline}
\label{5.30.1a}
\int_\Omega \big(\nabla \varphi,A\nabla y_\e\big)_{\mathbb{R}^N}\chi_{\Omega_\e}\,dx + \int_{\Omega} \big(\nabla \psi,A\nabla h \big)_{\mathbb{R}^N}\,dx\\
-\int_\Omega f_\e\varphi \chi_{\Omega_\e}\,dx -\left< v,\varphi\right>_{H^{-\frac{1}{2}}(\Gamma_\e);H^{\frac{1}{2}}(\Gamma_\e)}=0.
\end{multline}
holds true
for all $h\in L(A)$, $\varphi\in C^\infty_0(\Omega)$, and $\psi\in C^\infty_0(\Omega)$.
\end{definition}

Since for every $A\in \mathfrak{A}_{ad}$ and
$h\in D(A)$ the bilinear form $[h,\varphi]_A$ can be extended by continuity (see \eqref{2.8a}) onto the entire space $H^1_0(\Omega)$, it follows that the integral identity \eqref{5.30.1a} can be rewritten as follows
\begin{align}
\notag
\int_\Omega \big(\nabla \varphi,&A^{sym}\nabla y_\e+A^{skew}\nabla y_\e\big)_{\mathbb{R}^N}\chi_{\Omega_\e}\,dx\\ \notag &+ \int_{\Omega} \big(\nabla \psi,A^{sym}\nabla h\big)_{\mathbb{R}^N}\,dx +[h,\psi]_A
-\int_\Omega f_\e\varphi \chi_{\Omega_\e}\,dx\\ &-\left< v,\varphi\right>_{H^{-\frac{1}{2}}(\Gamma_\e);H^{\frac{1}{2}}(\Gamma_\e)}=0\quad \forall\, \varphi,\psi\in H^1_0(\Omega),\ \forall\,h\in L(A). \label{5.30.3a}
\end{align}
Hence, using the skew-symmetry property of the  matrix $A^{skew}\in L^2\big(\Omega;\mathbb{S}^N_{skew}\big)$ and the fact that the set $L(A)$ is closed with respect to the strong topology of $H^1_0(\Omega)$, we conclude: for every $\e>0$ there exist an element $h_\e$ in $L(A)$ such that the relation \eqref{5.30.3a} can be reduced to the following energy equality
\begin{align}
\notag
\int_\Omega \left(\nabla y_\e, A^{sym}y_\e\right)_{\mathbb{R}^N}\chi_{\Omega_\e}\,dx &+
\int_{\Omega} \big(\nabla y_\e,A^{sym}\nabla h_\e\big)_{\mathbb{R}^N}\,dx +[h_\e,y_\e]_A\\
&=\int_\Omega f_\e y_\e \chi_{\Omega_\e}\,dx +\left< v,y_\e\right>_{H^{-\frac{1}{2}}(\Gamma_\e);H^{\frac{1}{2}}(\Gamma_\e)}.
\label{5.30.4a}
\end{align}

Thus, in contrast to the "typical"\ energy equality to the boundary value problem \eqref{5.5b}, relation \eqref{5.30.4a} includes some extra term which coming from the singular energy of the boundary value problem \eqref{2.1}--\eqref{2.2} that was originally hidden in approximated problem \eqref{5.5b}.
However, in contrast to the similar functional effect for Hardy inequalities in bounded domains (see \cite{Vaz}),  the term $\int_{\Omega} \big(\nabla y_\e,A^{sym}\nabla h_\e \big)_{\mathbb{R}^N}\,dx + [ h_\e,y_\e]_A$ is additive to the total energy, and, hence, its influence may correspond to the increasing or decreasing of the total energy and may even constitute the main part of it.

\section{Optimality System for Regularized OCPs Associated with Perforated Domains $\Omega_\e$ and its Asymptotic Analysis}
\label{Sec 6}

As follows from Theorem~\ref{Th 5.6}, for each $\e>0$ small enough, the optimal control problem $\left<\inf_{(A,v,y)\in\Xi_\e}
I_\e(A,v,y)\right>$, where the cost functional $I_\e:\Xi_\e\rightarrow \mathbb{R}$ and its domain $\Xi_\e\subset \mathfrak{A}^\e_{ad}\times H^{-\frac{1}{2}}(\Gamma_\e)\times H^1_0(\Omega_\e;\partial\Omega)$ are defined by \eqref{5.5a}--\eqref{5.5b}, is a well-posed controllable system. Hence, to deduce an optimality system for this problem, we make use of the following well-know result.
\begin{theorem}[Ioffe and Tikhomirov \cite{IoTi,Fursik}]
\label{Th 6.1}
Let $Y$, $U$, and $V$ be Banach spaces, let $J:Y\times U\to \overline{\mathbb{R}}$ be a cost functional, let $F:Y\times U\to V$ be a mapping, and let $U_\partial$ be a convex subset of the space $U$ containing more than one point.
Let $(\widehat{u},\widehat{y})\in U\times Y$ be a solution to the problem
\begin{gather*}
J(u,y)\rightarrow\inf,\\
F(u,y)=0,\quad u\in U_\partial.
\end{gather*}
For each $u\in U_\partial$, let the mapping $y\mapsto J(u,y)$ and $y\mapsto F(u,y)$ be continuously differentiable for
$y\in \mathcal{O}(\widehat{y})$, where $\mathcal{O}(\widehat{y})$ is some neighbourhood of the point $\widehat{y}$, and let
$\mathrm{Im}\,F^\prime_y(\widehat{u},\widehat{y})$ be closed and it has a finite codimension in $V$. In addition, for $y\in \mathcal{O}(\widehat{y})$, let the function $u\mapsto J(u,y)$ be convex, the functional $J$ is G\^{a}teaus-differentiable with respect to $u$ at the point $(\widehat{u},\widehat{y})$, and the mapping $u\mapsto F(u,y)$ is continuous from $U$ to $Y$ and affine, i.e.,
\[
F(\gamma u_1+(1-\gamma) u_2,y)=\gamma F(u_1,y) + (1-\gamma) F(u_2,y),\quad\forall\,u_1,u_2 \in U, \gamma\in \mathbb{R}.
\]
Then there exists a pair $(\lambda,p)\in \left(R_{+}\times V^\ast\right)\setminus \{0\}$ such that
\begin{gather}
\label{6.1} \left<\mathcal{L}^\prime_y(\widehat{u},\widehat{y},\lambda,p),h\right>_{Y^\ast;Y}=0,\quad\forall\,h\in Y,\\
\label{6.2}
 \left<\mathcal{L}^\prime_u(\widehat{u},\widehat{y},\lambda,p),u\right>_{U^\ast;U}\ge 0,\quad\forall\,u\in U_\partial-\widehat{u},
\end{gather}
where the Lagrange functional $\mathcal{L}$ is defined by equality
\begin{equation}
\label{6.3}
\mathcal{L}(u,y,\lambda,p)=\lambda J(u,y)+\left<p,F(u,y)\right>_{V^\ast;V}.
\end{equation}
If $\mathrm{Im}\,F^\prime_y(\widehat{u},\widehat{y})=V$, then it can be assumed that $\lambda=1$ in \eqref{6.1}--\eqref{6.2}.
\end{theorem}

For our further analysis, we set
\begin{align}
\label{6.3a}
Y&=H^1_0(\Omega_\e;\partial\Omega),\quad V=L^2(\Omega_\e)\times H^{-\frac{1}{2}}(\Gamma_\e),\\
\label{6.3b}
U&=\big(L^2(\Omega;\mathbb{S}^N_{sym})\oplus
L^2(\Omega;\mathbb{S}^N_{skew})\big)\times H^{-\frac{1}{2}}(\Gamma_\e), \\
\label{6.3c}
U_\partial&=\mathfrak{A}_{ad}\times H^{-\frac{1}{2}}(\Gamma_\e):=\left(\mathfrak{A}_{ad,1}\oplus \mathfrak{A}_{ad,2}\right)\times H^{-\frac{1}{2}}(\Gamma_\e),\\
J&:=\left\|y-y_d\right\|^2_{L^2(\Omega_\e)}
 + \int_{\Omega_\e}\left(\nabla y, A^{sym}\nabla y\right)_{\mathbb{R}^N}\,dx + \frac{1}{\e^\sigma}\|v\|^2_{H^{-\frac{1}{2}}(\Gamma_\e)},\\
\label{6.3d}
F&(A,v,y)=\left(-\div\big(A\nabla y\big) - f_\e, \frac{\partial y}{ \partial \nu_{A}}-v\right).
\end{align}
Since for each $(g,w)\in L^2(\Omega_\e)\times H^{-\frac{1}{2}}(\Gamma_\e)$ the boundary value problem
\begin{gather}
\label{6.4.1}
-\div\big(A\nabla y\big) = g\quad\text{in }\ \Omega_\e,\\[1ex]
\label{6.4.2}
y=0\text{ on }\partial\Omega,\quad
\partial y/ \partial \nu_{A}=w\ \text{on }\Gamma_\e
\end{gather}
has a unique solution $y\in H^1_0(\Omega_\e;\partial\Omega)$ \cite{Lions_Mag:72}, we have $\mathrm{Im}\,F^\prime_y(\widehat{u},\widehat{y})=V$. Thus, the assumptions of Theorem~\ref{Th 6.1} are
obviously satisfied. It means that the Lagrange functional $\mathcal{L}_\e$  to the optimal control problem $\left<\inf_{(A,v,y)\in\Xi_\e}
I_\e(A,v,y)\right>$ can be defined by formula (with $\lambda=1$ in \eqref{6.1}--\eqref{6.2})
\begin{align}
\notag
\mathcal{L}_\e&(A,v,y,p,p_1)=\left\|y-y_d\right\|^2_{L^2(\Omega_\e)}
 + \int_{\Omega_\e}\left(\nabla y, A^{sym}\nabla y\right)_{\mathbb{R}^N}\,dx
+ \frac{1}{\e^\sigma}\|v\|^2_{H^{-\frac{1}{2}}(\Gamma_\e)}\\
&+ \left(-\div\big(A\nabla y\big) - f_\e, p_1\right)_{L^2(\Omega_\e)} +
\left< \frac{\partial y}{ \partial \nu_{A}}-v,p_{\,2}\,\right>_{H^{-\frac{1}{2}}(\Gamma_{\e});H^{\frac{1}{2}}(\Gamma_{\e})},
\label{6.4}
\end{align}
where $p=(p_1,p_{\,2})\in V^\ast:= L^2(\Omega_\e)\times H^{\frac{1}{2}}(\Gamma_{\e})$.

Let $\gamma^0_{\Gamma_\e}:H^1_0(\Omega_\e;\partial\Omega)\rightarrow H^{\frac{1}{2}}(\Gamma_{\e})$ be the trace operator, i.e.
$\gamma^0_{\Gamma_\e}$ is the extension by continuity of the restriction operator $\gamma^0_{\Gamma_\e}( u)=u\big|_{\Gamma_\e}$ given for all $u\in C_0^\infty(\mathbb{R}^N)$.
We are now in a position to prove the following result.
\begin{theorem}
\label{Th 6.5}
For a given $\e>0$, let
$$
(A_\e^0, v_\e^0, y_\e^0)\in \big(L^2(\Omega;\mathbb{S}^N_{sym})\oplus
L^2(\Omega;\mathbb{S}^N_{skew})\big)\times H^{-\frac{1}{2}}(\Gamma_\e)\times H^1_0(\Omega_\e;\partial\Omega)
$$
be an optimal solution to the regularized problems \eqref{5.4}--\eqref{5.5b}. Assume that the following condition holds true
\begin{equation}
\label{6.5.0}
\div\,\left(\left(A_\e^0\right)^{skew}\nabla y_\e^0\right) \in L^2(\Omega_\e).
\end{equation}
Then there exists an element $p_\e\in H^1_0(\Omega_\e;\partial\Omega)$ such that the tuple
$$(A_\e^0, v_\e^0, y_\e^0,p_\e,\gamma^0_{\Gamma_\e}(p_\e))$$ satisfies the following system of relations
\begin{align}
\label{6.5}
-\div\big(A_\e^0\nabla y_\e^0\big) =\ & f_\e\quad\text{in }\ \Omega_\e,\quad
y_\e^0= 0\quad \text{ on }\partial\Omega,\\
\label{6.5b}
\partial y_\e^0/ \partial \nu_{A_\e^0}=\ & v_\e^0\quad \text{on }\Gamma_\e,\\
 \div\left(\left(A_\e^0\right)^t\nabla p_\e\right)=\ &- 2\,\div\left(\left(A_\e^0\right)^{sym}\nabla y^0_\e\right)+ 2\left(y^0_\e-y_d\right),
\   \text{a.e. in }\ \Omega_\e,\label{6.6}\\
p_\e=\ & 0\quad \text{ on }\partial\Omega,\quad
\label{6.6b}
\partial p_\e^0/ \partial \nu_{(A_\e^0)^t}=\  0\quad \text{on }\Gamma_\e,\\
\label{6.7a}
v^0_\e =\ & \frac{\e^\sigma}{2}\Lambda_{H^{\frac{1}{2}}(\Gamma_{\e})}\gamma^0_{\Gamma_\e}(p_\e),\\
\label{6.7b}
\int_{\Omega_\e}\big(\nabla y^0_\e+\nabla p_\e, &\left(A^{sym}-(A^0_\e)^{sym}\right)\nabla y^0_\e\big)_{\mathbb{R}^N}\,dx\\
+ \int_{\Omega_\e}\big(\nabla p_\e, &\left(A^{skew}-(A^0_\e)^{skew}\right)\nabla y^0_\e\big)_{\mathbb{R}^N}\,dx
\ge 0,\quad\forall\, A\in \mathfrak{A}_{ad},
\end{align}
where $\Lambda_{H^{\frac{1}{2}}(\Gamma_{\e})}$ is the canonical isomorphism of $H^{\frac{1}{2}}(\Gamma_{\e})$ onto $H^{-\frac{1}{2}}(\Gamma_{\e})$.
\end{theorem}
\begin{remark}
It is worth to notice that, in contrast to \eqref{6.5}, relation \eqref{6.6} should be interpreted as an equality of $L^2$-functions. It means that the description of boundary value problem \eqref{6.6}--\eqref{6.6b} in the sense of distributions takes other form, namely,
\begin{align*}
 \div\left(\left(A_\e^0\right)^t\nabla p_1\right)=\ & 2\left(f_\e+\,\div\left(\left(A_\e^0\right)^{skew}\nabla y^0_\e\right)+\left(y^0_\e-y_d\right)\right),\ \text{ in }\ \Omega_\e,\\
p_\e=\ & 0\quad \text{ on }\partial\Omega,\quad
\partial p_\e^0/ \partial \nu_{(A_\e^0)^t}= \partial y_\e^0/ \partial \nu_{(A_\e^0)^{skew}}\quad \text{on }\Gamma_\e,
\end{align*}
where the component $\partial y_\e^0/ \partial \nu_{(A_\e^0)^{skew}}$ is unknown a priori.
Here, we have used the fact that
\begin{equation}
\label{6.7d}
-\div\left(\left(A_\e^0\right)^{sym}\nabla y^0_\e\right)=f_\e+\div\left(\left(A_\e^0\right)^{skew}\nabla y^0_\e\right)\quad\text{in }\ \Omega_\e
\end{equation}
by equation \eqref{6.5}.
\end{remark}

\begin{proof}
By Theorem~\ref{Th 6.1}, there exists a pair $p=(p_1,p_{\,2})\in V^\ast:= L^2(\Omega_\e)\times H^{\frac{1}{2}}(\Gamma_{\e})$ such that the Lagrange functional $\mathcal{L}$ satisfies relations \eqref{6.1}--\eqref{6.2}. The direct computations show that, in view of \eqref{6.4}, the condition \eqref{6.1} takes the form
\begin{multline}
\label{6.8}
\left\langle\mathcal{D}_y\,\widehat{L}_\e(A_\e^0, v_\e^0, y_\e^0,p_1,p_2),h\right\rangle_{Y^\ast;Y}=
2\int_{\Omega_\e}\left(\nabla h, \left(A_\e^0\right)^{sym}\nabla y^0_\e\right)_{\mathbb{R}^N}\,dx\\
+2
\int_{\Omega_\e}\left(y^0_\e-y_d\right) h\,dx
 +
\left< \frac{\partial h}{ \partial \nu_{A_\e^0}},p_{\,2}\,\right>_{H^{-\frac{1}{2}}(\Gamma_{\e});H^{\frac{1}{2}}(\Gamma_{\e})}\\
 -\int_{\Omega_\e}
\div\big(A_\e^0\nabla h\big)p_1\,dx=0,\quad \forall\,h\in H^2(\Omega_\e)\cap H^1_0(\Omega_\e;\partial\Omega)
\end{multline}
(here we have used the fact that $\mathrm{Im}\,F^\prime_y(\widehat{u},\widehat{y})=V$). As follows from \eqref{6.8} and \eqref{6.5.0}, for $h\in C^\infty_0(\Omega_\e)$, we have
\begin{multline}
\label{6.9}
2\int_{\Omega_\e}\left(\nabla h, \left(A_\e^0\right)^{sym}\nabla y^0_\e\right)_{\mathbb{R}^N}\,dx
+2
\int_{\Omega_\e}\left(y^0_\e-y_d\right) h\,dx\\- \int_{\Omega_\e}\div\left(\left(A_\e^0\right)^t\nabla p_1\,\right) h\,dx=
-2\int_{\Omega_\e} \div\left(\left(A_\e^0\right)^{sym}\nabla y^0_\e\right) h\,dx\\
+2
\int_{\Omega_\e}\left(y^0_\e-y_d\right) h\,dx
 - \int_{\Omega_\e}\div\left(\left(A_\e^0\right)^t\nabla p_1\,\right) h\,dx=0.
\end{multline}
Due to equality \eqref{6.7d} and the initial assumptions \eqref{6.5.0}, relation \eqref{6.9} implies that $\div\left(\left(A_\e^0\right)^t\nabla p_1\right)\in L^2(\Omega_\e)$. Hence,  $\left(A_\e^0\right)^t\nabla p_1\in H(\Omega_\e;\div)$, where
\[
H(\Omega_\e;\div)=\left\{\xi\ |\ \xi\in L^2(\Omega_\e;\mathbb{R}^N),\  \div\xi\in L^2(\Omega_\e)\right\}.
\]
Thanks to Lipschitz properties of $\partial\Omega_\e$, we can conclude that (see, for instance, \cite{Lions_Mag:72,Cioran})
$\partial p_1/ \partial \nu_{(A_\e^0)^t}\in H^{-\frac{1}{2}}(\partial\Omega_{\e})$ and the map
\[
\left(A_\e^0\right)^t\nabla p_1\in H(\Omega_\e;\div)\ \mapsto \frac{\partial p_1}{ \partial \nu_{(A_\e^0)^t}}\in H^{-\frac{1}{2}}(\partial\Omega_{\e})
\]
is linear and continuous. Moreover, if $\left(A_\e^0\right)^t\nabla p_1\in H(\Omega_\e;\div)$ and $h\in H^2(\Omega_\e)\cap H^1_0(\Omega_\e;\partial\Omega)$, then  the Green formula
\begin{align}
\notag
-\int_{\Omega_\e}\div & \big(A_\e^0\nabla h\big) p_1\,dx=-\int_{\Omega_\e}\div \big(\left(A_\e^0\right)^t\nabla p_1\,\big) h\,dx\\
&-\left< \frac{\partial h}{ \partial \nu_{A_\e^0}},\gamma^0_{\partial\Omega_\e}\big(p_1\big)\,\right>_{H^{-\frac{1}{2}}(\Omega_{\e});H^{\frac{1}{2}}(\Omega_{\e})}
+ \left< \frac{\partial p_1}{ \partial \nu_{(A_\e^0)^t}},h\right>_{H^{-\frac{1}{2}}(\Gamma_{\e});H^{\frac{1}{2}}(\Gamma_{\e})}
\label{6.10}
\end{align}
is valid.  Then, combining this relation with \eqref{6.8}--\eqref{6.9}, we arrive at the following identity
\begin{multline}
\label{6.11}
\left\langle\mathcal{D}_y\,\widehat{L}_\e(A_\e^0, v_\e^0, y_\e^0,p_1,p_2),h\right\rangle_{Y^\ast;Y}=
\left< \frac{\partial p_1}{ \partial \nu_{(A_\e^0)^t}},h\,\right>_{H^{-\frac{1}{2}}(\Gamma_{\e});H^{\frac{1}{2}}(\Gamma_{\e})}\\ -
\left< \frac{\partial h}{ \partial \nu_{A_\e^0}},\gamma^0_{\partial\Omega_\e}\big(p_1\big)\,\right>_{H^{-\frac{1}{2}}(\Omega_{\e});H^{\frac{1}{2}}(\Omega_{\e})}+
\left< \frac{\partial h}{ \partial \nu_{A_\e^0}},p_{\,2}\,\right>_{H^{-\frac{1}{2}}(\Gamma_{\e});H^{\frac{1}{2}}(\Gamma_{\e})}=0,
\end{multline}
which is valid for all $h\in H^2(\Omega_\e)\cap H^1_0(\Omega_\e;\partial\Omega)$ and all $p=(p_1,p_2)$ such that
\begin{equation}
\label{6.12}
\begin{array}{c}
p_1\ \text{satisfies \eqref{6.9}},\\[1ex]
(p_1,p_{\,2})\in L^2(\Omega_\e)\times H^{\frac{1}{2}}(\Gamma_{\e})\ \text{ and }\
\left(A_\e^0\right)^t\nabla p_1\in H(\Omega_\e,\mathrm{div}).
\end{array}
\end{equation}

As follows from \eqref{6.11}, for each $
h\in C^\infty_0(\mathbb{R}^N;\Gamma_\e)\cap C_0(\mathbb{R}^N;\partial\Omega)\subset H^2(\Omega_\e)\cap H^1_0(\Omega_\e;\partial\Omega),
$
we have
\[
\left< \frac{\partial h}{ \partial \nu_{A_\e^0}},\gamma^0_{\partial\Omega}\big(p_1\big)\,\right>_{H^{-\frac{1}{2}}(\partial \Omega);H^{\frac{1}{2}}(\partial \Omega)}=0.
\]
Since $C^\infty_0(\mathbb{R}^N;\Gamma_\e)\cap C_0(\mathbb{R}^N;\partial\Omega)$ is dense in $H^{-\frac{1}{2}}(\partial \Omega)$ and
the matrix $\left(A_\e^0\right)^{sym}$ is positive defined, it follows that
\begin{equation}
\label{6.13}
\gamma^0_{\partial\Omega}\big(p_1\big)=0.
\end{equation}
Hence, equality \eqref{6.11}, for all $h\in C^\infty_0(\mathbb{R}^N;\Gamma_\e)$, gives
\begin{equation}
\label{6.14}
\left< \frac{\partial h}{ \partial \nu_{A_\e^0}},p_{\,2}\,\right>_{H^{-\frac{1}{2}}(\Gamma_{\e});H^{\frac{1}{2}}(\Gamma_{\e})}-
\left< \frac{\partial h}{ \partial \nu_{A_\e^0}},\gamma^0_{\Gamma_\e}\big(p_1\big)\,\right>_{H^{-\frac{1}{2}}(\Gamma_{\e});H^{\frac{1}{2}}(\Gamma_{\e})}=0.
\end{equation}
Taking into account the fact that the mapping
$$\partial/ \partial \nu_{A_\e^0}: H^2(\Omega_\e)\cap H^1_0(\Omega_\e;\partial\Omega)\rightarrow H^{\frac{1}{2}}(\Gamma_{\e})$$
is an epimorphism (see Theorem~1.1.4 in \cite{Fursik}), from \eqref{6.14} it follows that
\begin{equation}
\label{6.15}
\gamma^0_{\Gamma_\e}\big(p_1\big)=p_{\,2}.
\end{equation}

Thus, in view of \eqref{6.13} and \eqref{6.15}, relation \eqref{6.11} takes the form
\[
\left\langle\mathcal{D}_y\,\widehat{L}_\e(A_\e^0, v_\e^0, y_\e^0,p_1,\gamma^0_{\Gamma_\e}\big(p_{1}\big)),h\right\rangle_{Y^\ast;Y}=
\left< \frac{\partial p_1}{ \partial \nu_{(A_\e^0)^t}},h\,\right>_{H^{-\frac{1}{2}}(\Gamma_{\e});H^{\frac{1}{2}}(\Gamma_{\e})}=0
\]
for all $h\in H^2(\Omega_\e)\cap H^1_0(\Omega_\e;\partial\Omega)$.
Applying the same arguments as before, we finally conclude that
\begin{equation}
\label{6.16}
\frac{\partial p_1}{ \partial \nu_{(A_\e^0)^t}}=0\quad\text{on }\ \Gamma_\e\ \text{ (in the sense of distribution)}.
\end{equation}

As a result, having gathered relations \eqref{6.9}, \eqref{6.13}, and \eqref{6.16}, we arrive at the boundary value problem
\eqref{6.6}--\eqref{6.6b}. Moreover, by the regularity of solutions to the problem \eqref{6.6}--\eqref{6.6b}, we have
$p_\e\in H^2(\Omega_\e)\cap H^1_0(\Omega_\e;\partial\Omega)$ \cite{Gilbarg}.

In order to end of the proof of this theorem, it remains to show the validity of the relations \eqref{6.7a}--\eqref{6.7b}.
With that in mind, we note that, in view of the structure \eqref{6.3a}--\eqref{6.3c}, condition \eqref{6.2} takes the form
\begin{multline}
\left(\mathcal{D}_A \mathcal{L}(A_\e^0, v_\e^0, y_\e^0,p_\e,\gamma^0_{\Gamma_\e}(p_\e)),A-A_\e^0\right)_{L^2(\Omega;\mathbb{M}^N)}\ge 0,\quad\forall\, A\in \mathfrak{A}^\e_{ad}\quad\Longrightarrow\\
\int_{\Omega_\e}\left(\nabla y^0_\e+\nabla p_\e, \left(A^{sym}-(A^0_\e)^{sym}\right)\nabla y^0_\e\right)_{\mathbb{R}^N}\,dx\\
+ \int_{\Omega_\e}\left(\nabla p_\e, \left(A^{skew}-(A^0_\e)^{skew}\right)\nabla y^0_\e\right)_{\mathbb{R}^N}\,dx
\ge 0,\quad\forall\, A\in \mathfrak{A}^\e_{ad},
\label{6.17}
\end{multline}
\begin{equation}
\label{6.18}
\mathcal{D}_v \mathcal{L}(A_\e^0, v_\e^0, y_\e^0,p_\e,\gamma^0_{\Gamma_\e}(p_\e))=0\quad\Longrightarrow
\frac{2}{\e^\sigma} v^0_\e -\Lambda_{H^{\frac{1}{2}}(\Gamma_{\e})}\gamma^0_{\Gamma_\e}(p_\e)=0,
\end{equation}
 Here, we have used the fact that $H^{\frac{1}{2}}(\Gamma_{\e})$ can be reduced to a Hilbert space with respect to an appropriate equivalent norm, and, hence, $H^{-\frac{1}{2}}(\Gamma_{\e})$ is a dual Hilbert space as well (for the details we refer to Lions and Magenes \cite[p.35]{Lions_Mag:72}).
\end{proof}
\begin{remark}
\label{Rem 6.4}
In view of the assumption \eqref{6.5.0}, we make use of the following observation.
Let $\left\{(A_\e,v_\e,y_\e)\in  \Xi_\e\right\}_{\e>0}$ be a weakly convergent sequence in the sense of Definition~\ref{Def 5.10}. Since in this case $\left\{y_\e\in H^1_0(\Omega_\e;\partial\Omega)\right\}_{\e>0}$ are the solutions to the boundary value problem \eqref{6.4.1}--\eqref{6.4.2} with $A=A_\e$, and $g=f_\e\in L^2(\Omega)$, and $w=v_\e\in H^{-\frac{1}{2}}(\Gamma_\e)$, it follows that
the sequence $\left\{\div\big(A_\e\nabla y_\e\big)\chi_{\Omega_\e}\right\}_{\e>0}$ is obviously bounded in $L^2(\Omega)$. However, because of the non-symmetry of $L^2$-matrices $\left\{A_\e\right\}_{\e>0}$, it does not imply the same property for the sequence $\left\{\div\big(A^{skew}_\e\nabla y_\e\big)\chi_{\Omega_\e}\right\}_{\e>0}$. In order to guarantee this property, we make use of the notion of divergence $\div A$ of a skew-symmetric matrix $A\in L^2\big(\Omega;\mathbb{S}^N_{skew}\big)$. We define it as a vector-valued distribution
$d\in H^{-1}(\Omega;\mathbb{R}^N)$ following the rule
\begin{equation}
\label{6.0a}
\left<d_i,\varphi\right>_{H^{-1}(\Omega);H^1_0(\Omega)}= -\int_\Omega
(a_i,\nabla\varphi)_{\mathbb{R}^N}\,dx,\  \forall\,\varphi\in
C^\infty_0(\Omega),\ \forall\,i\in\left\{1,\dots,N\right\},
\end{equation}
where $a_i$ stands for the $i$-th column of the matrix $A$. As a result, we can give the following conclusion:
if $\div A_\e^{skew}\in L^\infty(\Omega;\mathbb{R}^N)$ for all $\e>0$ and the sequence $\left\{\div A_\e^{skew}\right\}_{\e>0}$ is uniformly bounded in $L^\infty(\Omega;\mathbb{R}^N)$, then there exists a constant $C>0$ independent of $\e$ such that
\begin{equation}
\label{6.0b}
\sup_{\e>0} \left\|\chi_{\Omega_\e} \div\big(A^{skew}_\e\nabla y_\e\big)\right\|_{L^2(\Omega)}\le C.
\end{equation}
Indeed, since
\begin{align*}
-&\big<\mathrm{div}\,\left(A_\e^{skew}\nabla \psi_\e\right),\chi_{\Omega_\e}\varphi\big>_{H^{-1}(\Omega);H^1_0(\Omega)}
=
-\left<\mathrm{div}\,\left(A_\e^{skew}\nabla \psi_\e\right),\varphi\right>_{H^{-1}(\Omega_\e);H^1_0(\Omega_\e)}\\
&=\Big< \mathrm{div}\,\left[
\begin{array}{c}
a^t_{1,\e}\nabla \psi_\e\\
\cdots \\a^t_{N,\e}\nabla \psi_\e
\end{array}\right],\varphi\Big>_{H^{-1}(\Omega_\e);H^1_0(\Omega_\e)}
=
\sum_{i=1}^N\left<\mathrm{div}\,a_{i,\e}, \varphi\frac{\partial\psi_\e}{\partial x_i}\right>_{H^{-1}(\Omega_\e);H^1_0(\Omega_\e)}\\
&+ \underbrace{\int_{\Omega_\e}
\sum_{i=1}^N \sum_{j=1}^N \left(a_{ij,\e} \frac{\partial^2 \psi_\e}{\partial x_i\partial x_j}\right) \varphi\,dx}_{=0\ \atop {\text{ since }\ A_\e^{skew}\in L^2(\Omega;\mathbb{S}^N_{skew})}}=\int_{\Omega_\e} \left(\,\div A_\e^{skew},\nabla \psi_\e\right)_{\mathbb{R}^N} \varphi\,dx,\\
\end{align*}
for any $\psi_\e,\varphi\in C^\infty_0(\Omega_\e)$ (due to the fact that $\div A_\e^{skew}\in L^\infty(\Omega;\mathbb{R}^N)$ for all $\e>0$), it follows that this relation can be extended by continuity to the following one
\[
-\left<\mathrm{div}\,\left(A_\e^{skew}\nabla y_\e\right),\chi_{\Omega_\e}\varphi\right>_{H^{-1}(\Omega);H^1_0(\Omega)}
=\int_{\Omega_\e} \left(\,\div A_\e^{skew},\nabla y_\e\right)_{\mathbb{R}^N} \varphi\,dx.
\]
Hence,
\begin{align*}
\left\|\chi_{\Omega_\e} \div\big(A^{skew}_\e\nabla y_\e\big)\right\|_{L^2(\Omega)}&\le
(\mathcal{L}^N(\Omega))^{1/2}\|\div A_\e^{skew}\|_{L^\infty(\Omega;\mathbb{R}^N)}\\
&\times \|\nabla y_\e\|_{L^2(\Omega_\e;\mathbb{R}^N)}<+\infty.
\end{align*}
To deduce the estimate \eqref{6.0b}, it remains to refer to the boundedness of $y_\e$ in variable $H^1(\Omega_\e;\partial\Omega)$ (see Definition~\ref{Def 5.10}).
\end{remark}

Our next intention is to provide an asymptotic analysis of the optimality system \eqref{6.5}--\eqref{6.7b} as $\e$ tends to zero. With that in mind, we assume the fulfilment of the following Hypotheses:
\begin{enumerate}
\item[(H1)] For each admissible control $A\in \mathfrak{A}_{ad}$  the corresponding bilinear form $[y,\varphi]_A$ is continuous in the following sense:
\begin{equation}
\label{6.19.a}
\lim_{\e\to 0} [y_\e,p_\e]_A = [y,p\,]_A
\end{equation}
provided
$\left\{p_\e\right\}_{\e>0}\subset H^1_0(\Omega)$, $\left\{y_\e\right\}_{\e>0}\subset H^1_0(\Omega)$, $y_\e\rightharpoonup y$ in $H^1_0(\Omega)$, $p_\e\rightarrow p$ in $H^1_0(\Omega)$, and  $y, y_\e\in D(A)$ for $\e>0$ small enough.
\item[(H2)] Let $\left\{(A_\e^0, v_\e^0, y_\e^0, p_\e\,)\right\}_{e>0} $ be a sequence of tuples such that, for each $\e>0$ the corresponding cortege $(A_\e^0, v_\e^0, y_\e^0, p_\e\,)$ satisfies the optimality system \eqref{6.5}--\eqref{6.7b}. Then there exists a sequence of extension operators
    $$\left\{P_\e\in \mathcal{L}\left(H^1_0(\Omega_\e;\partial\Omega),H^1_0(\Omega)\right)\right\}_{\e>0}$$
    and element $\overline{\psi}\in H^1_0(\Omega)$ such that
\[
P_\e(p_\e) \rightarrow \overline{\psi}\quad\text{strongly in }\ H^1_0(\Omega)\quad\text{and}\quad \overline{\psi}\in D(A^\ast).
\]
\end{enumerate}

\begin{theorem}
\label{Th 6.19} Let $y_d\in L^2(\Omega)$ and $f\in
H^{-1}(\Omega)$ be given distributions. Let $A^\ast\in L^2\big(\Omega;\mathbb{S}^N_{skew}\big)$ be a matrix of the $\mathfrak{F}$-type. Let $\left\{(A^0_\varepsilon,v^0_\e,y^0_\e)\in \Xi_\varepsilon\right\}_{\e>0}$ be a sequence of optimal solutions to regularized problems \eqref{5.4}--\eqref{5.5b}, and let $(A^0,y^0)\in D(A^\ast)\times H^1_0(\Omega)$ be its $w$-limit. Let $\left\{p^0_\e\in H^1_0(\Omega_\e;\partial\Omega)\right\}_{\e>0}$ be a sequence of corresponding adjoint states. Then, the fulfilment of the Hypotheses (H1)--(H2) implies that $(A^0,y^0)\in \mathfrak{A}_{ad}\times H^1_0(\Omega)$ is an optimal pair to the original OCP \eqref{2.3}--\eqref{2.3a} and there exists an element $\overline{\psi}\in H^1_0(\Omega)$ such that
\begin{gather}
\label{6.20}
(A^0_\varepsilon,v^0_\e,y^0_\e)\,\stackrel{w}{\rightarrow}\, (A^0,y^0)\ \text{ as }\ \e\to 0,\\
\label{6.21}
P_\e(p_\e) \rightarrow \overline{\psi}\quad\text{strongly in }\ H^1_0(\Omega),\\
\label{6.22}
-\div\big(A^0\nabla y^0\big) =\  f\quad\text{in }\ \Omega,\quad
y=0\quad\text{on }\ \partial\Omega,
\\
\begin{split}
\div\left(\left(A^0\right)^t\nabla \overline{\psi}\right)&=\ - 2\,\div\left(\left(A^0\right)^{sym}\nabla y^0\right)+ 2\left(y^0-y_d\right)
\   \text{ in }\ \Omega,\\
\overline{\psi}&=0\quad\text{on }\ \partial\Omega,
\end{split}
\label{6.23}\\
\int_\Omega \big(\nabla y^0,\big(A^{sym}-\left(A^0\right)^{sym}\big)\left(\nabla y^0+\nabla\overline{\psi}\right)\big)_{\mathbb{R}^N}\,dx\notag\\
\ge\   [y^0,\overline{\psi}]_{A^0}-[y^0,\overline{\psi}]_{A},\ \forall A\in \mathfrak{A}_{ad},
\label{6.24}
\end{gather}
\end{theorem}
\begin{proof}
To begin with, we note that due to Theorem~\ref{Th 5.31}, the sequence of optimal solutions $\left\{(A^0_\varepsilon,v^0_\e,y^0_\e)\in \Xi_\varepsilon\right\}_{\e>0}$ to the regularized problems \eqref{5.4}--\eqref{5.5b}
is compact with respect to $w$-convergence and each of its $w$-cluster pairs $(A^0,y^0)$ is an optimal pair to the original problem \eqref{2.3}--\eqref{2.3a}. Hence, $(A^0,y^0)\in \mathfrak{A}_{ad}$, and the limit passage in \eqref{6.5}--\eqref{6.5b}
as $\e\to 0$ leads us to the relation \eqref{6.22} in the sense of distributions. In what follows, we divide the proof onto several steps.

Step~1. Since the integral identity
\begin{align}
\notag
&\int_{\Omega} \big(\nabla \varphi,\left(A_\e^0\right)^{sym}\nabla P_\e(p_\e)-\left(A_\e^0\right)^{skew}\nabla P_\e(p_\e)\big)_{\mathbb{R}^N}\chi_{\Omega_{\e}}\,dx\\
&=-2\int_{\Omega}\left(\nabla \varphi, \left(A_\e^0\right)^{sym}\nabla P_\e(y^0_\e)\right)_{\mathbb{R}^N}\chi_{\Omega_{\e}}\,dx
-2
\int_{\Omega}\left(P_\e(y^0_\e)-y_d\right) \varphi\chi_{\Omega_{\e}}\,dx
\label{6.25}
\end{align}
holds true for every $\e>0$ and $\varphi\in C^\infty_0(\Omega)$, we can pass to the limit in \eqref{6.25} as $\e\to 0$ due to Hypothesis (H2) and Definition \ref{Def 5.10} (here, we apply the arguments of Remark~\ref{Rem 5.3.1}). Using the strong convergence $\chi_{\Omega_{\e}}\rightarrow \chi_\Omega$ in $L^2(\Omega)$ (see Proposition~\ref{Prop 5.3}), we arrive at the equality
\begin{align}
\notag
\int_{\Omega} \big(\nabla \varphi,&\left(A^0\right)^{t}\nabla \overline{\psi}\big)_{\mathbb{R}^N}\,dx
=-2\int_{\Omega}\left(\nabla \varphi, \left(A^0\right)^{sym}\nabla y^0\right)_{\mathbb{R}^N}\,dx\\
&-2
\int_{\Omega}\left(y^0-y_d\right) \varphi\,dx,\quad
 \forall\,\varphi\in C^\infty_0(\Omega).
\label{6.25.a}
\end{align}
Hence, $\overline{\psi}\in D(A^0)\subset H^1_0(\Omega)$ (see Proposition~5 in \cite{HK_1}) and $\overline{\psi}$ satisfies relation \eqref{6.23} in the sense of distributions.

Step~2. On this step we study the limit passage in inequality \eqref{6.7b} as $\e\to 0$. To this end, we rewrite it as follows
\begin{equation}
\label{6.26}
J_1^\e(A)\ge J_2^\e - J_3^\e(A),\quad\forall\, A\in \mathfrak{A}^\e_{ad},\ \forall\,\e>0,
\end{equation}
where
\begin{align}
\label{6.27}
J_1^\e(A)=\ & \int_{\Omega_\e}\big(\nabla y^0_\e, A^{sym}\nabla y^0_\e\big)_{\mathbb{R}^N}\,dx,\\
\label{6.28}
J_2^\e=\ & \int_{\Omega_\e}\big(\nabla y^0_\e, (A^0_\e)^{sym}\nabla y^0_\e\big)_{\mathbb{R}^N}\,dx,\\
\label{6.29}
J_3^\e(A)=\ & \int_{\Omega_\e}\left(\nabla y^0_\e, \left(A^{t}-(A^0_\e)^{t}\right)\nabla p_\e\right)_{\mathbb{R}^N}\,dx.
\end{align}
By Theorem~\ref{Th 5.31} (see \eqref{5.33}), we have
\begin{align}
\notag
I\left(A^0,y^0\right):=\ & \left\|y^0-y_d\right\|^2_{L^2(\Omega)} + \int_{\Omega}\left(\nabla y^0, \left(A^0\right)^{sym}\nabla y^0\right)_{\mathbb{R}^N}\,dx\\
\notag =\ &\lim_{\e\to 0} I_{\e}(A^0_\e,v^0_\e,y^0_\e):= \lim_{\e\to 0} \left\|(y_\e^0-y_d)\chi_{\Omega_\e}\right\|^2_{L^2(\Omega)}\\
&+ \lim_{\e\to 0} \int_{\Omega_\e}\left(\nabla y_\e^0, \left(A_\e^0\right)^{sym}\nabla y_\e^0\right)_{\mathbb{R}^N}\,dx
+ \lim_{\e\to 0} \frac{1}{\e^\sigma}\|v_\e^0\|^2_{H^{-\frac{1}{2}}(\Gamma_\e)}.
\label{6.30}
\end{align}
Since
\begin{equation}
\label{6.30.aa}
\lim_{\e\to 0} \left\|(y_\e^0-y_d)\chi_{\Omega_\e}\right\|^2_{L^2(\Omega)}=\left\|y^0-y_d\right\|^2_{L^2(\Omega)}
\end{equation}
by the compactness of the embedding $H^1_0(\Omega)\hookrightarrow L^2(\Omega)$, and $\lim_{\e\to 0} \e^{-\sigma}\|v_\e^0\|^2_{H^{-\frac{1}{2}}(\Gamma_\e)}=0$ by Theorem~\ref{Th 5.31} (see estimate \eqref{5.30.3}), it follows from \eqref{6.30} that
\begin{equation}
\label{6.31}
\lim_{\e\to 0}J_2^\e = \int_{\Omega}\left(\nabla y^0, \left(A^0\right)^{sym}\nabla y^0\right)_{\mathbb{R}^N}\,dx=:J_2.
\end{equation}

Step~3. As for the term $J_3^\e(A)$, we see that
\begin{align}
\notag
\lim_{\e\to 0}J_3^\e(A)=\ &\lim_{\e\to 0} \int_{\Omega_\e}\left(\nabla y^0_\e, (A^0_\e)^{t}\nabla p_\e\right)_{\mathbb{R}^N}\,dx=(\ \text{by \eqref{6.25}}\ )\\
\notag=\ &\lim_{\e\to 0}\Big[-2\int_{\Omega}\left(\nabla P_\e(y^0_\e), \left(A_\e^0\right)^{sym}\nabla P_\e(y^0_\e)\right)_{\mathbb{R}^N}\chi_{\Omega_{\e}}\,dx\\
\notag&-2\int_{\Omega}\left(P_\e(y^0_\e)-y_d\right) P_\e(y^0_\e)\chi_{\Omega_{\e}}\,dx\Big] =
(\ \text{by \eqref{6.31} and \eqref{6.30.aa}}\ )\\
\notag=\ &
-2\int_{\Omega}\left(\nabla y^0, \left(A^0\right)^{sym}\nabla y^0\right)_{\mathbb{R}^N}\,dx
-2  \int_{\Omega}\left(y^0-y_d\right) y^0\,dx\\
\notag=\ &\lim_{\e\to 0}\Big[-2\int_{\Omega}\left(\nabla P_\e(y^0_\e), \left(A^0\right)^{sym}\nabla y^0\right)_{\mathbb{R}^N}\chi_{\Omega_{\e}}\,dx\\
\notag&-2  \int_{\Omega}\left(y^0-y_d\right) P_\e(y^0_\e)\chi_{\Omega_{\e}}\,dx\Big]=(\ \text{by \eqref{6.25.a}}\ )\\
\notag=\ &
\lim_{\e\to 0} \int_{\Omega} \big(\nabla P_\e(y^0_\e),\left(A^0\right)^{t}\nabla \overline{\psi}\big)_{\mathbb{R}^N}\chi_{\Omega_{\e}}\,dx\\
\notag=\ & \int_{\Omega} \big(y^0,\left(A^0\right)^{sym}\nabla \overline{\psi}\big)_{\mathbb{R}^N}\,dx +\lim_{\e\to 0} [P_\e(y^0_\e)\chi_{\Omega_{\e}},\overline{\psi}\,]_{A^0}=(\text{by (H2)})\\
\label{6.32}=\ &
\int_{\Omega} \big(y^0,\left(A^0\right)^{sym}\nabla \overline{\psi}\big)_{\mathbb{R}^N}\,dx + [y^0,\overline{\psi}\,]_{A^0}
\end{align}
and
\begin{align}
\notag
\lim_{\e\to 0} \int_{\Omega_\e}\left(\nabla y^0_\e, A^{t}\nabla p_\e\right)_{\mathbb{R}^N}\,dx&=
\int_{\Omega}\left(\nabla y^0, A^{sym}\nabla \overline{\psi}\right)_{\mathbb{R}^N}\,dx\\
&+\lim_{\e\to 0} \int_{\Omega}\left(\nabla P_\e(p_\e), A^{skew}\nabla P_\e(y^0_\e)\right)_{\mathbb{R}^N}\chi_{\Omega_\e}\,dx
\label{6.33}
\end{align}
as the limit of product of weakly and strongly convergence sequences in $L^2(\Omega;\mathbb{R}^N)$. Hence, combining relations \eqref{6.32} and \eqref{6.33}, we get
\begin{align}
\notag
\lim_{\e\to 0}&J_3^\e(A) =\int_\Omega \big(y^0,\left(A^{sym}-\left(A^0\right)^{sym}\right)\nabla \overline{\psi}\big)_{\mathbb{R}^N}\,dx - [y^0,\overline{\psi}\,]_{A^0}\\
\notag
&+ \lim_{\e\to 0} \int_{\Omega}\left(\nabla P_\e(p_\e), A^{skew}\nabla P_\e(y^0_\e)\right)_{\mathbb{R}^N}\chi_{\Omega_\e}\,dx=(\text{by Hypotheses (H1)--(H2)})\\
=\ &\int_\Omega \big(y^0,\left(A^{sym}-\left(A^0\right)^{sym}\right)\nabla \overline{\psi}\big)_{\mathbb{R}^N}\,dx - [y^0,\overline{\psi}\,]_{A^0}+[y^0,\overline{\psi}\,]_{A}=:J_3(A).
\label{6.34}
\end{align}

Step~4. At this step we study the asymptotic behaviour of the term $J_1^\e(A)$ in \eqref{6.27} as $\e\to 0$.
To this end, we note that in view of the property \eqref{1.1},  the lower semicontinuity of $L^2$-norm with respect to the weak convergence, immediately leads us to the inequality
\begin{align}
\notag
\lim_{\e\to 0}&J_1^\e(A)=\liminf_{\e\to 0} \int_{\Omega_{\e}}\left(\nabla y_\e^0, A^{sym} \nabla y_\e^0\right)_{\mathbb{R}^N}\,dx\\
\notag
=\ &
\liminf_{\e\to 0} \|\chi_{\Omega_{\e}}\left(A^{sym}\right)^{1/2}\nabla y_\e^0\|^2_{L^2(\Omega;\mathbb{R}^N)}\\
\ge\ & \|\left(A^{sym}\right)^{1/2}\nabla y^0\|^2_{L^2(\Omega;\mathbb{R}^N)}=
\int_{\Omega}\left(\nabla y^0, A^{sym} \nabla y^0\right)_{\mathbb{R}^N}\,dx
= J_1(A).
\label{6.35}
\end{align}
However, because of inequality in \eqref{6.35}, we cannot assert that the limit values are related as follows
\begin{equation}
\label{6.36}
J_1(A)\ge J_2 - J_3(A),\quad\forall\, A\in \mathfrak{A}_{ad}.
\end{equation}
In order to guarantee this relation, we assume the converse, namely, there exists a matrix $A_\sharp\in \mathfrak{A}_{ad}$ such that
$J_1(A_\sharp)< J_2 - J_3(A_\sharp)$. That is, in view of \eqref{6.31},\eqref{6.34}, and \eqref{6.35}, this leads us to the relation
\begin{multline}
\int_{\Omega}\left(\nabla y^0, \left(A_\sharp^{sym} - \left(A^0\right)^{sym}\right)\nabla y^0\right)_{\mathbb{R}^N}\,dx \\
+\int_\Omega \big(y^0,\left(A_\sharp^{sym}-\left(A^0\right)^{sym}\right)\nabla \overline{\psi}\big)_{\mathbb{R}^N}\,dx< [y^0,\overline{\psi}\,]_{A^0}-[y^0,\overline{\psi}\,]_{A_\sharp}.
\label{6.37}
\end{multline}
The direct computations show that, in this case, we arrive at the inequality
\[
\widehat{L}(A_\sharp,y^0,1,\overline{\psi})< \widehat{L}(A^0,y^0,1,\overline{\psi})=I(A_0,y_0)=\inf_{(A,y)\in\Xi}I(A,y),
\]
where $\widehat{L}(A,y,\lambda,p)$ is the Lagrange function given by
\[
\widehat{L}(A,y,\lambda,p)=\lambda I(A,y)+\int_\Omega \big(\nabla p,A^{sym}\nabla y\big)_{\mathbb{R}^N}\,dx +[y,p\,]_A-\langle
f,p\rangle_{H^{-1}(\Omega);H^1_0(\Omega)}.
\]
However, this contradicts with the Lagrange principle, and therefore, the inequality \eqref{6.36} remains valid. Thus, following \eqref{6.36}, we finally get
\begin{equation*}
\int_{\Omega}\left(\nabla y^0, \left(A^{sym} - \left(A^0\right)^{sym}\right)(\nabla y^0+\nabla \overline{\psi})\right)_{\mathbb{R}^N}\,dx
\ge [y^0,\overline{\psi}\,]_{A^0}-[y^0,\overline{\psi}\,]_{A}
\end{equation*}
for all $A\in \mathfrak{A}_{ad}$. This concludes the proof.
\end{proof}

\begin{remark}
As Theorem~\ref{Th 6.19} indicates, the limit passage in optimality system \eqref{6.5}--\eqref{6.7b} for the regularized problems \eqref{5.4}--\eqref{5.5b} as $\e\to 0$ leads to the optimality system for the original OCP \eqref{2.3}--\eqref{2.3a}. However, a strict substantiation of this passage requires rather strong assumptions in the form of Hypotheses (H1)--H2). At the same time, the verification of these Hypotheses becomes trivial provided
\begin{gather}
\label{6.38}
A^\ast\in L^\infty(\Omega;\mathbb{S}^N_{skew})\quad\text{in  \eqref{2.3c}},\\
\label{6.39}
\text{and }\ \exists\,C>0\ :\ \|\div A^{skew}\|_{L^\infty(\Omega;\mathbb{R}^N)}\le C,\quad\forall\,A\in \mathfrak{A}_{ad}.
\end{gather}
Indeed,  in this case the relation \eqref{6.19.a} takes the form
\[
\lim_{\e\to 0} \int_{\Omega}\left(\nabla p_\e, A^{skew}\nabla y_\e\right)_{\mathbb{R}^N}\,dx= \int_{\Omega}\left(\nabla p\,, A^{skew}\nabla y\right)_{\mathbb{R}^N}\,dx
\]
and it holds obviously true provided
$y_\e\rightharpoonup y$ in $H^1_0(\Omega)$, $p_\e\rightarrow p$ in $H^1_0(\Omega)$, and
$A^{skew}\preceq A^\ast\in L^\infty(\Omega;\mathbb{S}^N_{skew})$. Hence, Hypothesis~(H1) is valid. As for Hypothesis~(H2), we see that  admissible controls $A\in \mathfrak{A}_{ad}$ with extra property \eqref{6.39} form a close set with respect to the strong convergence in $L^2(\Omega;\mathbb{S}^N_{skew})$. Moreover, in this case we have that the sequence
$\left\{\chi_{\Omega_\e}\div\,\left(\left(A_\e^0\right)^{skew}\nabla y_\e^0\right)\right\}_{\e>0}$ is uniformly bounded in $L^2(\Omega)$ (see Remark~\ref{Rem 6.4}). Hence, the sequence of adjoint states $\left\{p_\e\right\}_{\e>0}$, given by \eqref{6.6}--\eqref{6.6b}, is bounded in $H^2(\Omega_\e)$ by the regularity of solutions to the problem \eqref{6.6}--\eqref{6.6b}.
Hence, within a subsequence, we can suppose that the sequence $\left\{P_\e(p_\e)\right\}_{\e>0}$ is weakly convergent in $H^2(\Omega)$. This proves Hypothesis (H2).
\end{remark}

\section{Numerical simulations}\label{NumSim}
The main issue of this section is to present numerical simulations that tend to ascertain our approaches developed above. We restrict ourselves to the case when $\Omega$ is the unit ball of ${\mathbb R}^2$ or ${\mathbb R}^3$.

The numerical simulations have been conducted according three guidelines.

For this we consider some  matrix $A_d\in L^2(\Omega)^{N\times N}$ and $y_d$ in $H^1_0(\Omega)$, and set
$$f=f_d:=-{\div}(A_d\nabla y_d).$$
We focus on the following test case:

\begin{align}{\label{Prob_00}}
J_{{test}}(A,y)&:= \|y-y_d\|_{L^2(\Omega)}^2+\frac{\varepsilon_0}{2}\int_\Omega \left(\nabla (y-y_d),A^{{sym}}(y-y_d)\right)\,dx\longrightarrow\inf
\end{align}
subject to
\begin{align}
-\div(A\nabla y)&=f_d,  \label{Prob_0}\quad
y\in H_0^1(\Omega)
\end{align}
with the uniform ellipticity condition on $A^{{sym}}$ given by \eqref{1.1}. For this problem under view the algorithm used should allow to recover the pair $(A_d,y_d)$, because the minimum of \eqref{Prob_00} is clearly $0$.

Once validated, we return to the original OCP \eqref{0.1},  for which we consider singular
$y_d$  and  $A_d$ in two manners: we still consider $A_d$, $y_d$ and $f_d$ with $A_d$ possibly singular at some point $\xi$ of the unit ball $\Omega$  in ${\mathbb R}^2$ or ${\mathbb R}^3$.
We  triangulate $\Omega$  by a triangulation $\tau$ such that no vertices of $\tau$ is $\xi$ and such that no edges of $\tau$ contains $\xi$.

We proceed to the classical gradient algorithm.

In this case, we expect, but cannot prove, that the algorithm converges to a variational solution. Indeed, when projecting on the grid, due to our assumption, we cannot distinguish between singular and non singular data. Moreover, for each projected matrix $A$ in the admissible set, the projected matrix gives rise to a unique solution, thus the projected problem changes in its behavior. And of course as already said, due to the non-singular situation, we are led to think that the sequence of approximate solutions constructed will give rise to a variational solution.

In the final simulation procedure, we have punctured our domain and discretized the OCP given in \eqref{0.2}. Accordingly, there is now no singularity in the punctured domain. We, afterwards, consider refining the punctured domain by reducing the size of the hole.

In the following sections we describe more precisely each scheme and present some numerical results with some interpretations in each case that, we do think, clarifies the situation.

\subsection{Validation}

Throughout this section and the following ones, we will take $A_d$ of the following form:

In the 2d-case

\begin{equation}
\left\lbrace\begin{array}{ll}
\displaystyle{A_d = (1 + (r-1)^2) A_{sym} + \frac{0.1}{r^{0.5}} A_{asym}} \\
\mbox{ with } \\
A_{sym} = \begin{pmatrix}   1.& 0.2\\ 0.2& 1.1\end{pmatrix},\\
A_{asym} = \begin{pmatrix}   0.& 1.\\ -1.& 0.\end{pmatrix},
\end{array}
\right.
\end{equation}

whilst in the 3d-case

\begin{equation}
\left\lbrace\begin{array}{ll}
\displaystyle{A = (1 + (r-1)^2) A_{sym} + \frac{0.01}{r} A_{asym}} \\
\mbox{ with } \\
A_{sym} = \begin{pmatrix}  1.0& 0.2& 0.2\\ 0.2& 1.1& 0.2\\ 0.2& 0.2& 1.2\end{pmatrix},\\
A_{asym} = \begin{pmatrix}  0.&  1.& 1.\\  -1.&  0.& 1.\\  -1.& -1.& 0.\end{pmatrix}.
\end{array}
\right.
\end{equation}

For the case of the unpunctured domain, the gradient $$G_{test}:=\nabla_A J_{test}$$ is obtained by using the adjoint state $p$ (see, for instance, \eqref{6.17} and further). 

Let $p$ be the solution of
\begin{equation}{\label{adjoint}}\begin{split}
&-{\div}( A(x)^t \nabla p ) =  (y_d - y) + \varepsilon_0 {\div}( A(x)^{sym} \nabla(y_d - y)) \mbox{ in } \Omega,\\ &  p = 0 \mbox{ on } \partial \Omega.
 \end{split}
\end{equation}

 We get
 \begin{equation} \label{gradient}
  (G_{test}, W)_{L^2(\Omega;\mathbb{M}^N)} =  \int_{\Omega} \nabla y^t W \nabla p\, d\Omega + \frac{\varepsilon_0}{2} \int_{\Omega} \nabla(y - y_d)^t W \nabla(y - y_d)\, d\Omega,
 \end{equation}
 where $W\in L^2(\Omega)^{N\times N}$.

 We adopt a  finite element method for $y$ and $p$ such that $A$ is constant for each triangular element of the mesh. In order for the algorithm to be more efficient, we  use more data than these discrete components of $A$. We set $n$ different pairs $\left\lbrace(u_d^i, f_d^i), i = 1, n \right\rbrace$. To reduce the value of $n$, we choose to use a spatial smoothing for each component of $G_{test}$. In order to do so, several options are possible (\cite{Destuynder_06}, \cite{Faugeras_93}).
 The new cost functional modified according to these $n$ tests is now (with $y^i$ solution of the state equation (\ref{Prob_0}) for $f$ equals $f^i$ with $i=1,n$):
\begin{equation}
  \displaystyle{J(A,\{y^i,i=1,n\};\{y_d^i,i=1,n\}) = \frac{1}{n} \sum_{i=1,n} J^i(A,y^i;y_d^i)},
\end{equation}
where
$$
\displaystyle{J^i(A,y^i;y_d^i) = \frac{1}{2} \int_{\Omega} |y^i - y_d^i|^2 d\Omega + \frac{\varepsilon_0}{2} \int_{\Omega} \nabla(y^i - y_d^i)^T A^{sym} \nabla(y^i - y_d^i) d\Omega}.
$$
The gradient becomes hereafter a mean of terms obtained in \eqref{gradient}.\\

For the two-dimensional case, we use 16 pairs $(y_d^i, f_d^i)$ associated to a combination of sinusoidal functions useful to capture sufficient information. Each state $y_d^i$ verify the state problem with $f$ equal to $f_d^i$ and $A$ equal to the reference $A_{d}$ (Figure~\ref{fig:2D_whitout_hole_1_1}). The coefficient $\varepsilon_0$ is equal to $10^6$. The initial matrix $A$ is by its coefficients $$(A_{11}(x), A_{12}(x) ; A_{21}(x), A_{22}(x)) = (1, 0.2 ; 0.1, 1.1).$$  The results (Figure \ref{fig:2D_whitout_hole_1_2}) show a coherent convergence.

For the three-dimensional case, the simulation durations prevent to use the same level of discretization than for the two-dimensional cases. We use 48 pairs $(y_d^i, f_d^i)$ associated to a combination of sinusoidal functions equivalent to the 2D-cases. We use 11929 points and 72946 cells for the mesh (without hole).
So we work on 72946 variables for each component of $A$.

The number of pairs $(y_d^i, f_d^i)$ and the smoothing are useful and allow us to control all theses variables, but with difficulties.
We must parallelize our control problem. The $n$ pairs $(y_d^i, f_d^i)$ create $n$ different state problems, each of them can be computed on different core. We use this characteristic to reduce to a few days the simulation duration. We test our three-dimensional program with a singular asymmetric $A_{d}$. The results are shown on Figures \ref{fig:3D_sing_whithout_hole} and \ref{fig:3D_sing_whithout_hole_JG}. The results are as consistent as for the 2D-problem. 

\begin{figure}[h]
 \centering
 \includegraphics[width=6.0cm,clip=true,trim=10pt 30pt 10pt 25pt]{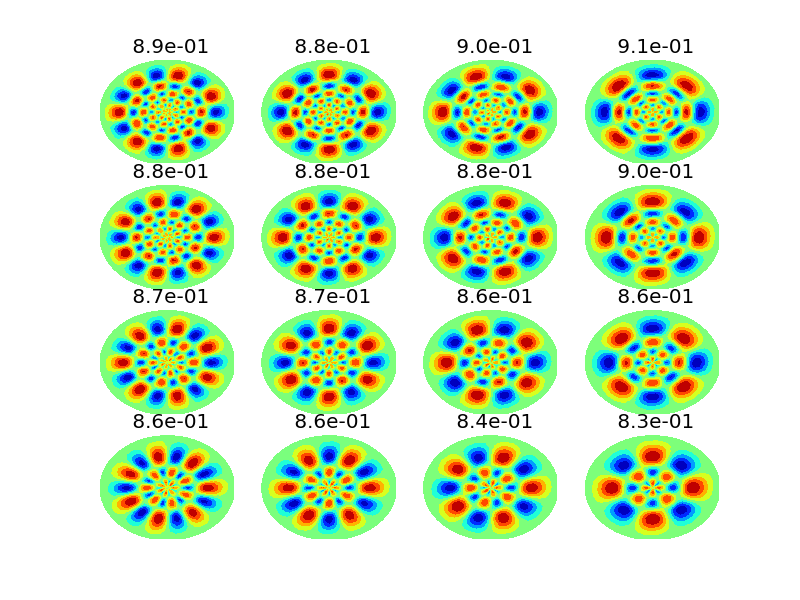}
 \includegraphics[width=6.0cm,clip=true,trim=30pt 20pt 30pt 30pt]{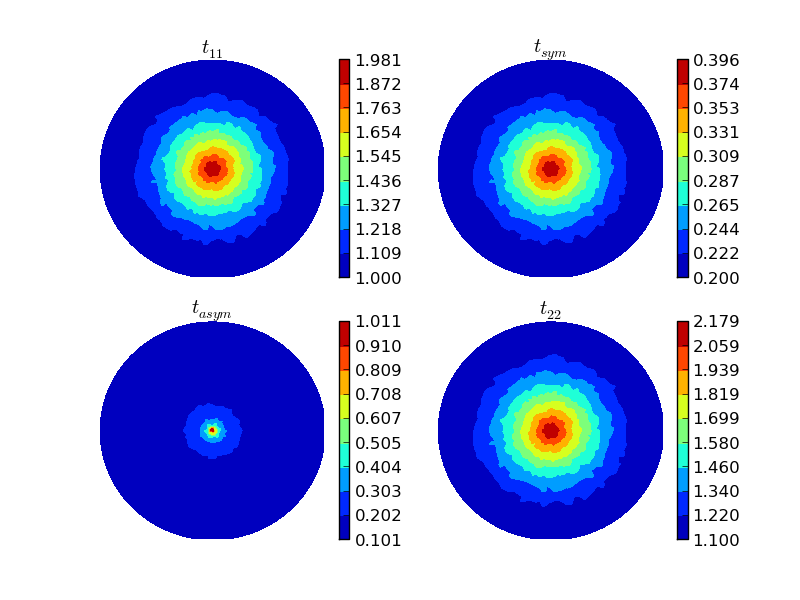}
 \caption{2D case - All $y_d^i$ (left) and $A_{d}$ with a singular asymmetric component (denoted $t_{...}$ in the picture).} \label{fig:2D_whitout_hole_1_1}
 ~\\~\\
 \includegraphics[width=6.0cm,clip=true,trim=30pt 20pt 30pt 30pt]{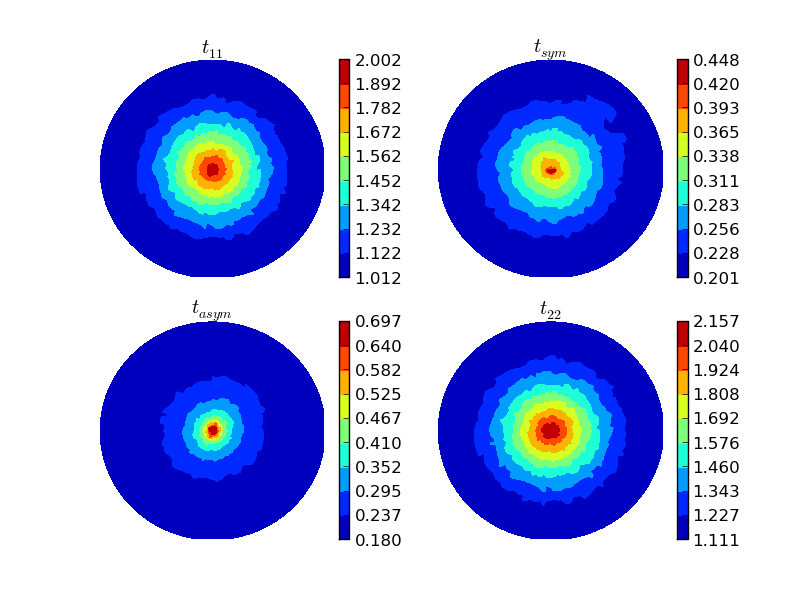}
 \includegraphics[width=6.0cm,clip=true,trim=30pt 25pt 30pt 25pt]{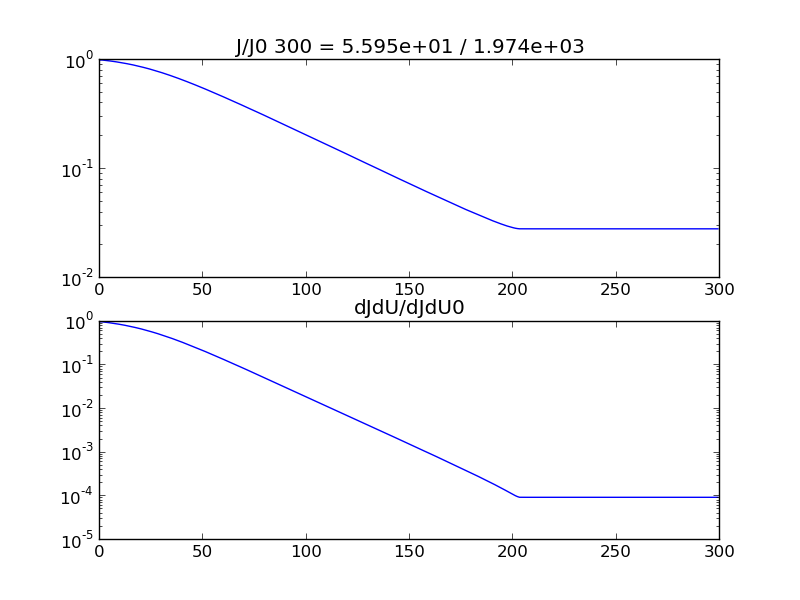}
 \caption{2D case - Final control $A$ (left) and relatives evolutions of $J$ and $||G||$.} \label{fig:2D_whitout_hole_1_2}
\end{figure}

\begin{figure}[h]
 \centering
 \includegraphics[width=1.95cm,bb=0 0 96 96,clip=true,trim=0pt 35pt 0pt 45pt]{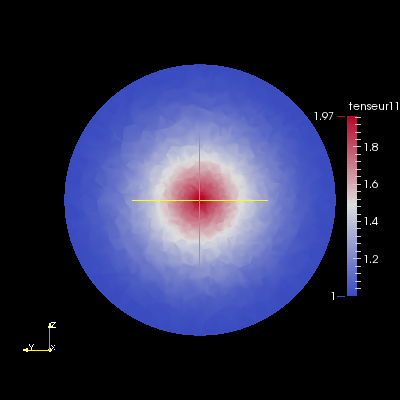}
 \includegraphics[width=1.95cm,bb=0 0 96 96,clip=true,trim=0pt 35pt 0pt 45pt]{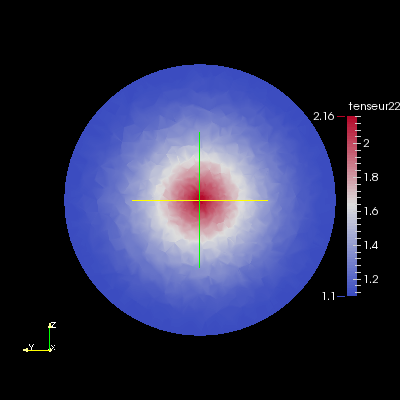}
 \includegraphics[width=1.95cm,bb=0 0 96 96,clip=true,trim=0pt 35pt 0pt 45pt]{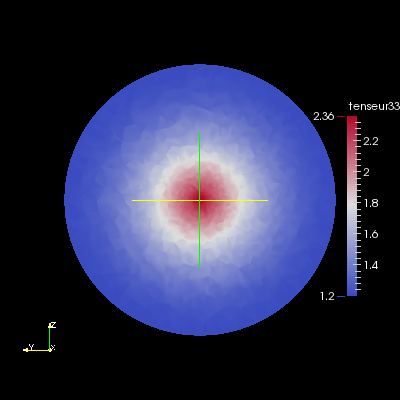} 
 \hspace{0.1cm}
 \includegraphics[width=1.95cm,bb=0 0 96 96,clip=true,trim=0pt 35pt 0pt 45pt]{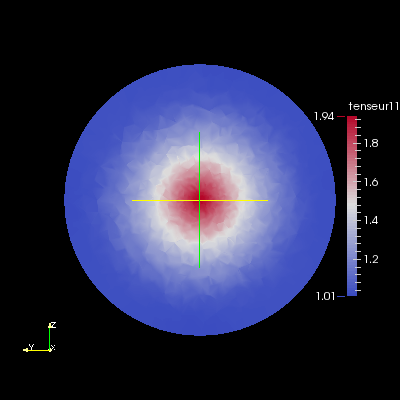}
 \includegraphics[width=1.95cm,bb=0 0 96 96,clip=true,trim=0pt 35pt 0pt 45pt]{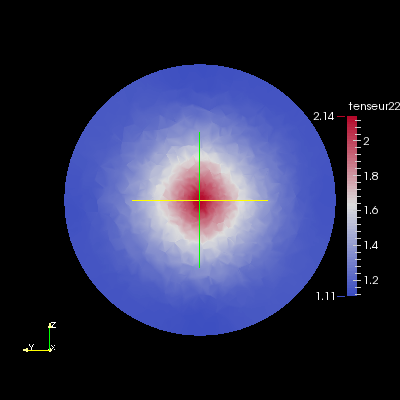}
 \includegraphics[width=1.95cm,bb=0 0 96 96,clip=true,trim=0pt 35pt 0pt 45pt]{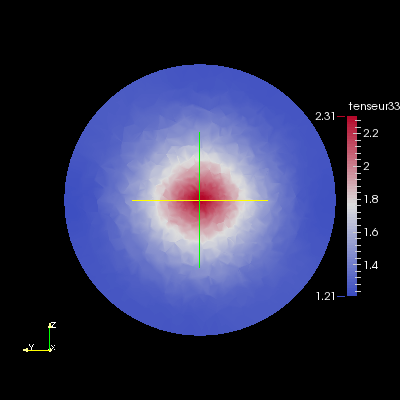}\\
 \includegraphics[width=1.95cm,bb=0 0 96 96,clip=true,trim=0pt 35pt 0pt 45pt]{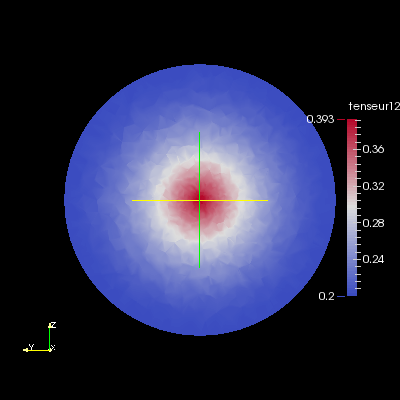}
 \includegraphics[width=1.95cm,bb=0 0 96 96,clip=true,trim=0pt 35pt 0pt 45pt]{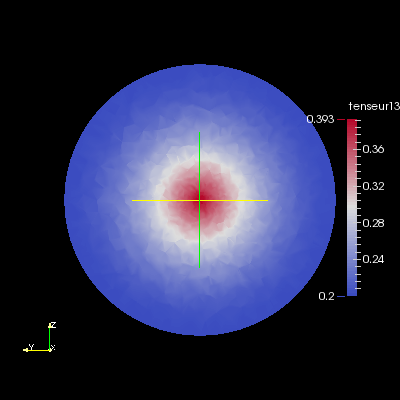}
 \includegraphics[width=1.95cm,bb=0 0 96 96,clip=true,trim=0pt 35pt 0pt 45pt]{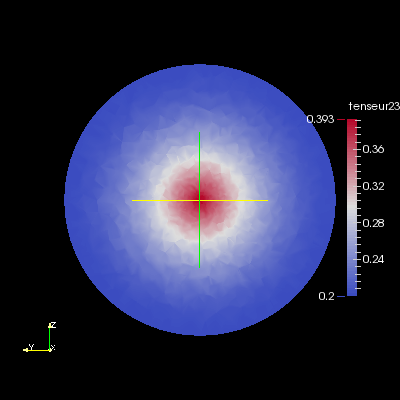}
 \hspace{0.1cm}
 \includegraphics[width=1.95cm,bb=0 0 96 96,clip=true,trim=0pt 35pt 0pt 45pt]{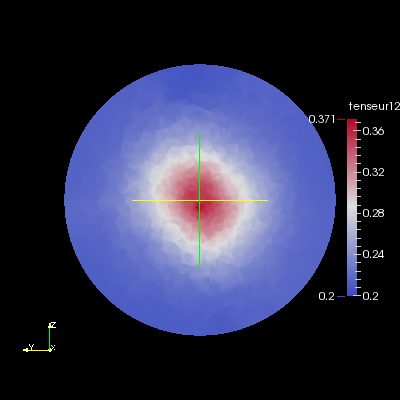}
 \includegraphics[width=1.95cm,bb=0 0 96 96,clip=true,trim=0pt 35pt 0pt 45pt]{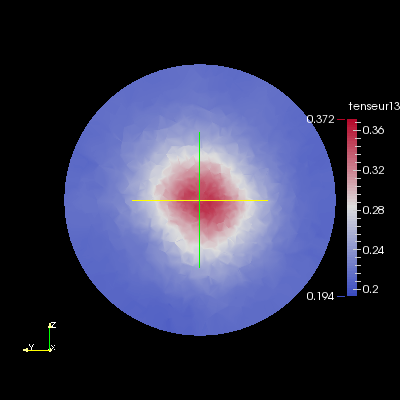}
 \includegraphics[width=1.95cm,bb=0 0 96 96,clip=true,trim=0pt 35pt 0pt 45pt]{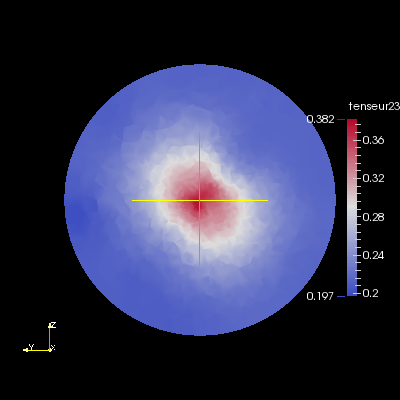}\\
 \includegraphics[width=1.95cm,bb=0 0 96 96,clip=true,trim=0pt 35pt 0pt 45pt]{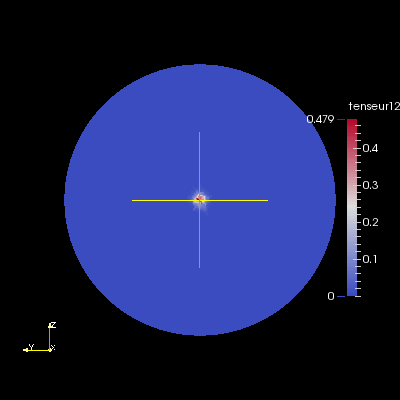}
 \includegraphics[width=1.95cm,bb=0 0 96 96,clip=true,trim=0pt 35pt 0pt 45pt]{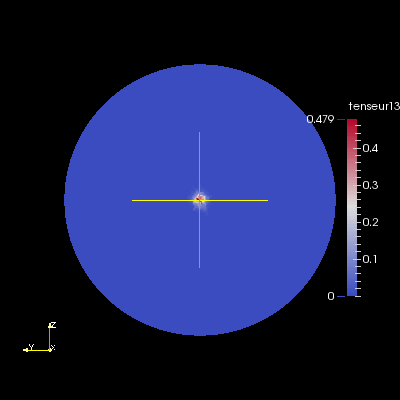}
 \includegraphics[width=1.95cm,bb=0 0 96 96,clip=true,trim=0pt 35pt 0pt 45pt]{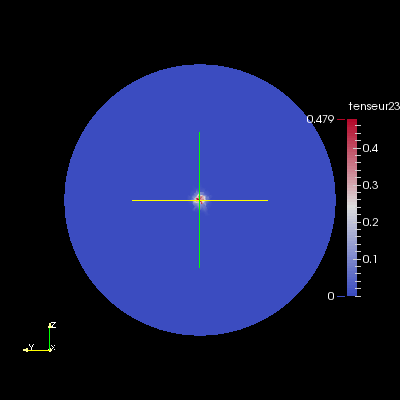}
 \hspace{0.1cm}
 \includegraphics[width=1.95cm,bb=0 0 96 96,clip=true,trim=0pt 35pt 0pt 45pt]{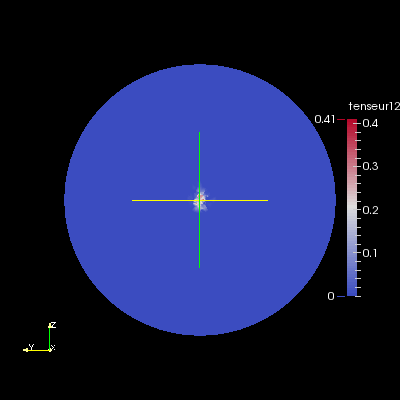}
 \includegraphics[width=1.95cm,bb=0 0 96 96,clip=true,trim=0pt 35pt 0pt 45pt]{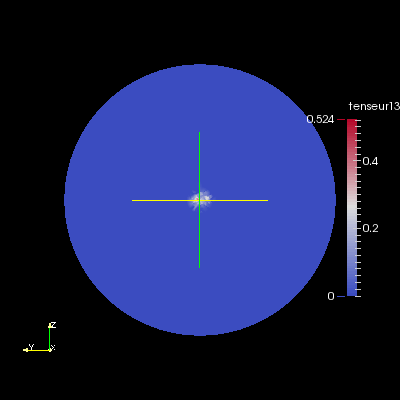}
 \includegraphics[width=1.95cm,bb=0 0 96 96,clip=true,trim=0pt 35pt 0pt 45pt]{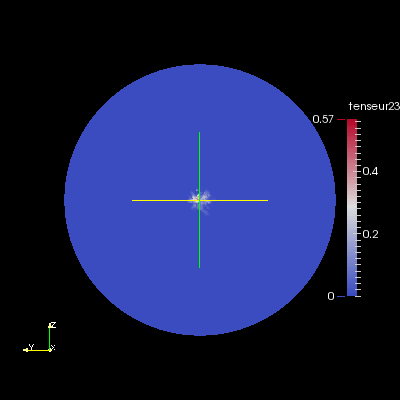}
 \caption{3D case: Components of $A_d$ (left) and the final control $A$ (right) for the plane (0,Y,Z), $A_{11}, A_{22}, A_{33}$ (line 1), symmetric part (line 2) and asymmetric part (line 3) of $A_{12}, A_{13}, A_{23}$ with singular asymmetric components.} \label{fig:3D_sing_whithout_hole}
 \end{figure}
 \begin{figure}[h]
 \centering
 \includegraphics[width=6.0cm,clip=true,trim=30pt 217pt 30pt 20pt]{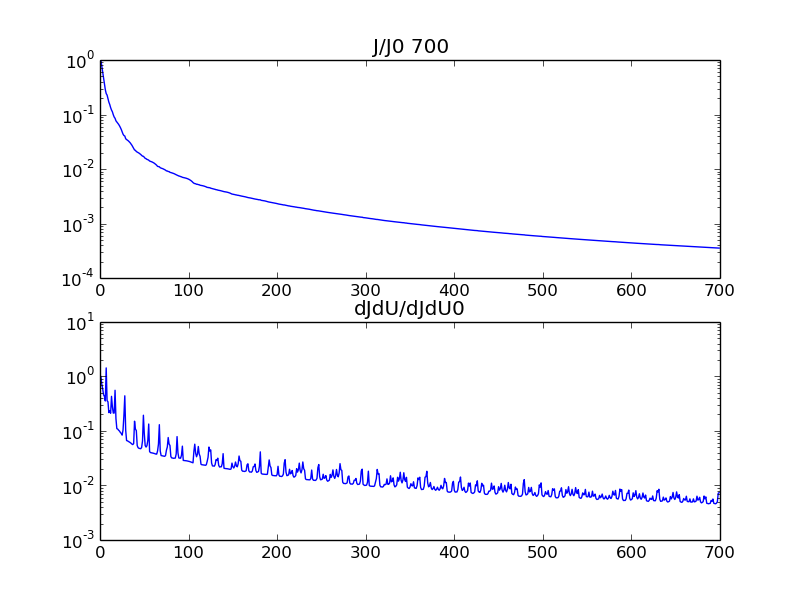}
 \includegraphics[width=6.0cm,clip=true,trim=30pt 28pt 30pt 216pt]{./Trace_3D_3_qcq_sing_V2_J_dJdU.png}
 \vspace{-0.25cm}\caption{3D case: Relatives evolutions of $J$ (left) and $||G||$.} \label{fig:3D_sing_whithout_hole_JG}
\end{figure}

\subsection{Discretization in the unpunctured domain}

We return to the original OCP \eqref{0.1}. We use the same pairs $\left\lbrace(y_d^i, f_d^i), i = 1, n \right\rbrace$ but now the real $A_{d}$ should be considered as unknown that is to say that we now consider a  real  optimization problem,  while the preceding test cases, could be considered as an identification or inverse problem.

The figures \ref{fig:qcqsing} and \ref{fig:3D_without_hole} show t\includegraphics[]{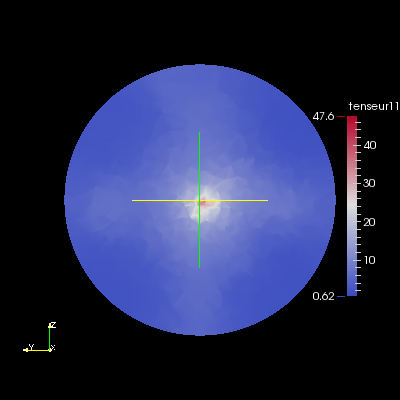}
he results of the direct simulation in two- and three-dimensional cases, respectively (without the trick with puncturing of the singularity region).  We use these results to compare with the next results associated to the OCP \eqref{0.2}.
\begin{figure}[h]
 \centering
 \includegraphics[width=5.5cm,bb=0 0 576 432,clip=true,trim=50pt 30pt 30pt 20pt]{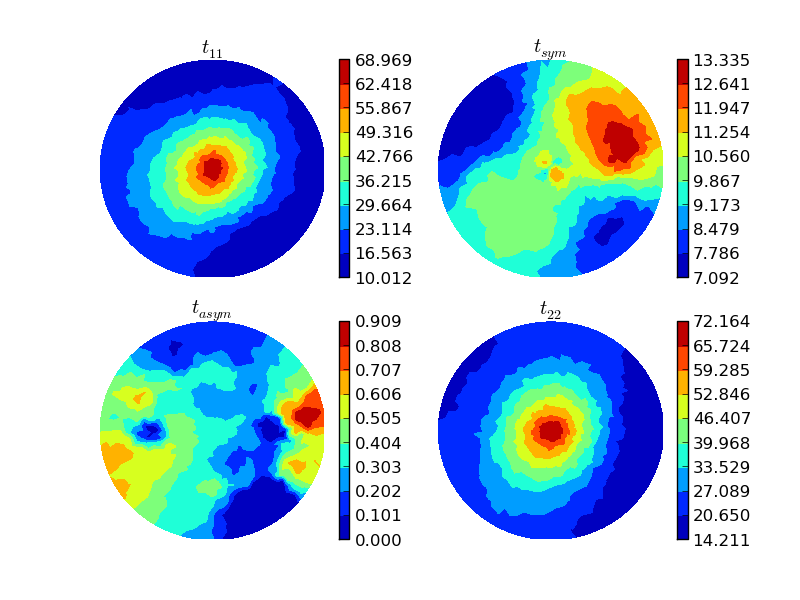}
 \includegraphics[width=5.5cm,bb=0 0 576 432,clip=true,trim=30pt 30pt 30pt 20pt]{./Tests-step2_Trace_2D_2_qcq_sing_V2_J_dJdU.png}
 \caption{2D case without hole: the components of $A$ (left), $J$ and $G$ (right).}
 \label{fig:qcqsing}
\end{figure}
\begin{figure}[h]
 \centering
\begin{tabular}{ll}
 \includegraphics[width=1.95cm,bb=0 0 96 96,clip=true,trim=0pt 35pt 0pt 45pt]{./Tests-step2_Trace_3D_3_qcq_sing_V2_ref2_025_fin_tenseur_37_tenseur11.png}
 \includegraphics[width=1.95cm,bb=0 0 96 96,clip=true,trim=0pt 35pt 0pt 45pt]{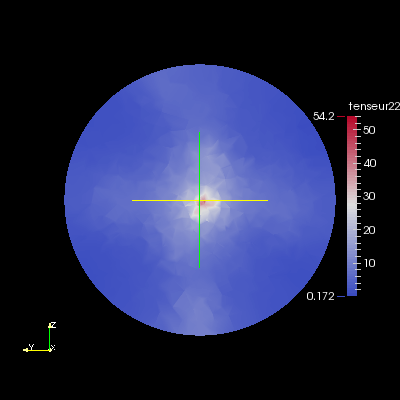}
 \includegraphics[width=1.95cm,bb=0 0 96 96,clip=true,trim=0pt 35pt 0pt 45pt]{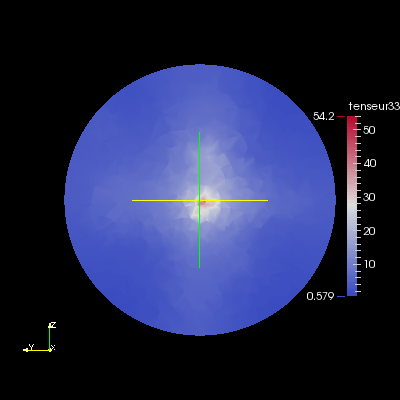}\\
 \includegraphics[width=1.95cm,bb=0 0 96 96,clip=true,trim=0pt 35pt 0pt 45pt]{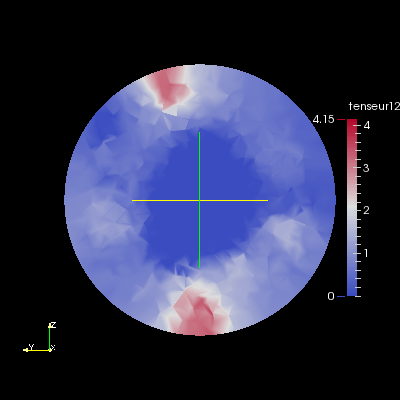}
 \includegraphics[width=1.95cm,bb=0 0 96 96,clip=true,trim=0pt 35pt 0pt 45pt]{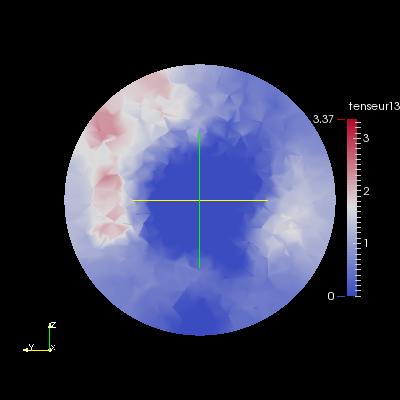}
 \includegraphics[width=1.95cm,bb=0 0 96 96,clip=true,trim=0pt 35pt 0pt 45pt]{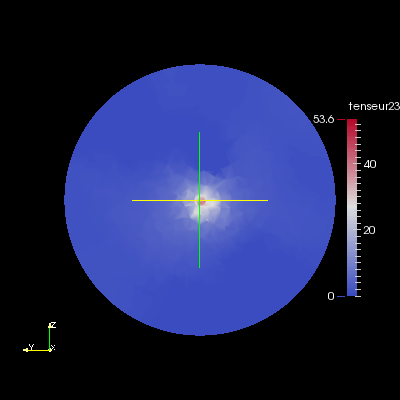}
 \includegraphics[width=5.5cm,clip=true,trim=30pt 217pt 30pt 20pt]{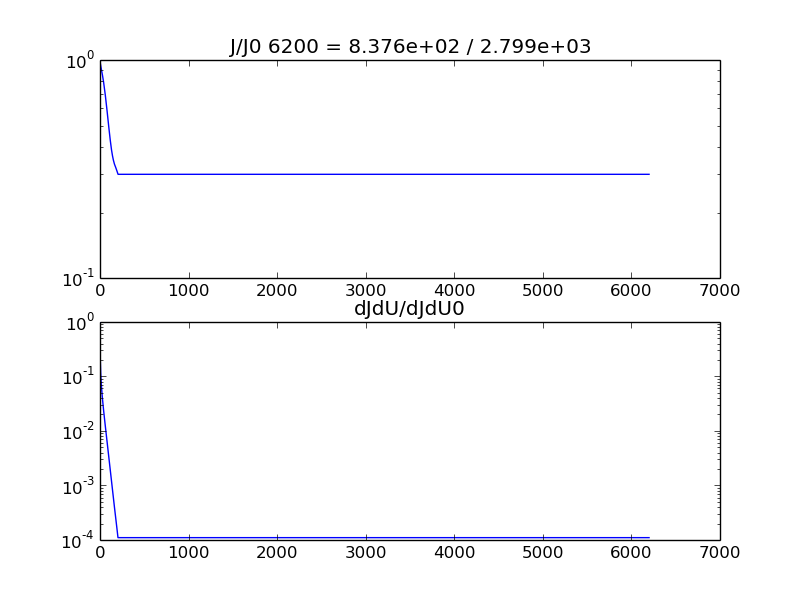}\\
 \includegraphics[width=1.95cm,bb=0 0 96 96,clip=true,trim=0pt 35pt 0pt 45pt]{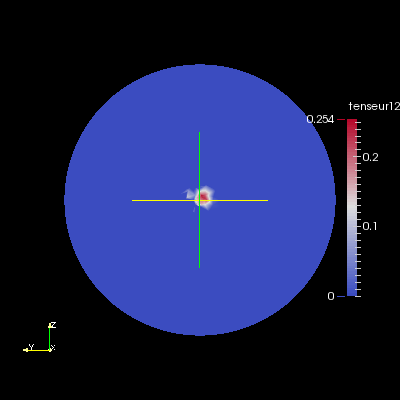}
 \includegraphics[width=1.95cm,bb=0 0 96 96,clip=true,trim=0pt 35pt 0pt 45pt]{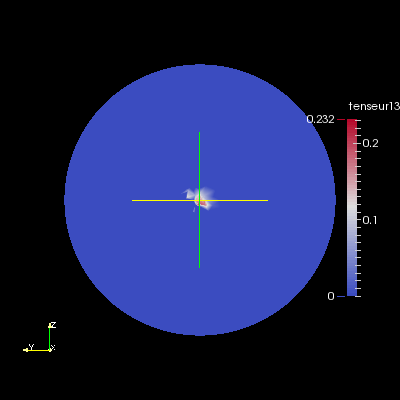}
 \includegraphics[width=1.95cm,bb=0 0 96 96,clip=true,trim=0pt 35pt 0pt 45pt]{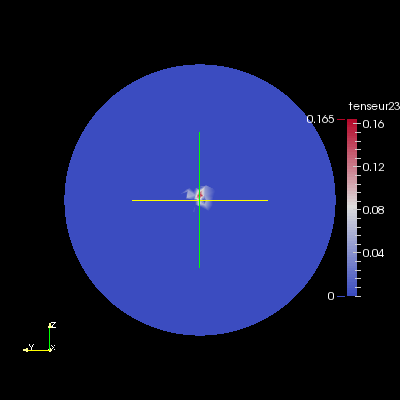}
 \includegraphics[width=5.5cm,clip=true,trim=30pt 28pt 30pt 216pt]{./Tests-step2_Trace_3D_3_qcq_sing_V2_ref2_025_fin_J_dJdU.png}\\
 \end{tabular}
 \caption{3D case without hole: (left) Visualisation of the components of the final control $A$ for the plane (0,Y,Z), $A_{11}, A_{22}, A_{33}$ (line 1), symmetric part of $A_{12}, A_{13}, A_{23}$  (line 2) and asymmetric part  of $A_{12}, A_{13}, A_{23}$ (line 3). (right) Relatives evolutions of $J$ (line 2) and $||G||$ (line 3).} \label{fig:3D_without_hole}
\end{figure}

\subsection{Discretization in the punctured domain}

At this step we consider the approximation of the original OCP in the form of \eqref{0.2}. In this case, we must add the $\hat{p}$ to the adjoint $p$ state solution of \eqref{adjoint} for each pairs $(u_d^i, f_d^i)$ where $\Omega$ is replaced by $\Omega_{\varepsilon}$ and
\begin{equation}
 \displaystyle{\hat{p} = -\frac{q}{\varepsilon^{\sigma}} \mbox{ on } \Gamma_{\varepsilon},}
\end{equation}
where $q$ satisfies (denoting $B_{\varepsilon}: $) 
\begin{equation}
\left\lbrace
\begin{array}{rl}
 q - \Delta q &= 0 \mbox{ in } B_{\varepsilon},\quad\left(\text{here }\ \Omega = \Omega_{\varepsilon} \bigcup B_{\varepsilon}\right) \\
 \displaystyle{\frac{\partial q}{\partial \nu}} &= v \mbox{ on } \Gamma_{\varepsilon} .
\end{array}
\right.
\end{equation}
We have then
\begin{equation}
 \displaystyle{\|v\|_{H^{-\frac{1}{2}}(\Gamma_{\varepsilon})} = \|q\|_{H^1(B_{\varepsilon})}}
\end{equation}

For the two-dimensional case, the pictures \ref{fig:qcqsingtrou05}, \ref{fig:qcqsingtrou025} show the results. The second case uses a smaller hole. For the three-dimensional case, the figure \ref{fig:3D_with_hole} shows the results. We can note that the values of the functional is always smaller than the cases without hole. For the second 2D case with a smaller hole, the components become more different than these obtained with the OCP \eqref{0.1}.

Of course these results do not validate the existence of variational and non-variational solutions. However, according to Zhkov [private communication], or if we believe that the uniqueness and regularity results in \cite{BrCa} lead to the absence of non-variational solutions in dimension $2$, the numerical simulations above tends to show that arguably this does exist in dimension $2$. However, due to computational performance and refinement requirement, it is probably very difficult to ascertain that our numerical simulations do prove the prevalence of non-variational solutions or not to OCP \eqref{0.1} on the class of admissible controls $A$ with unremovable singularity.

\begin{figure}[h]
 \centering
 \includegraphics[width=5.5cm,bb=0 0 576 432,clip=true,trim=50pt 30pt 30pt 20pt]{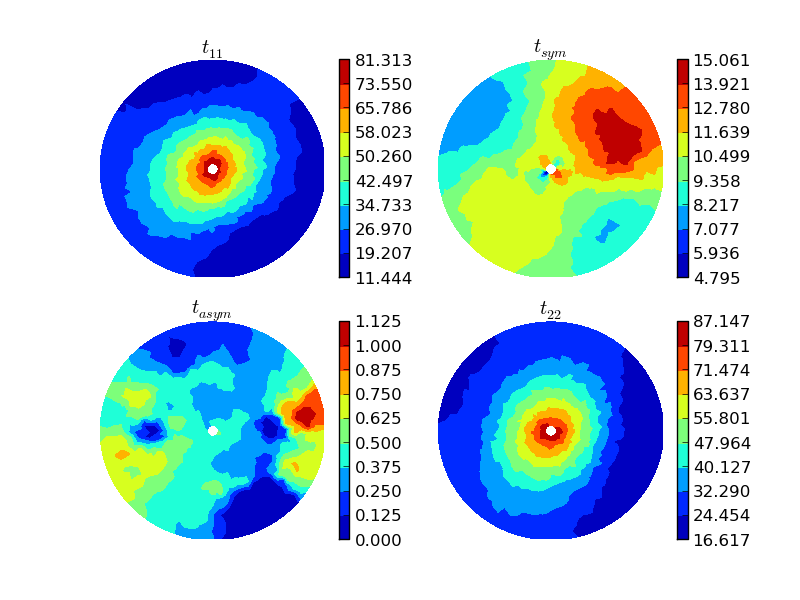}
 \includegraphics[width=5.5cm,bb=0 0 576 432,clip=true,trim=30pt 30pt 30pt 20pt]{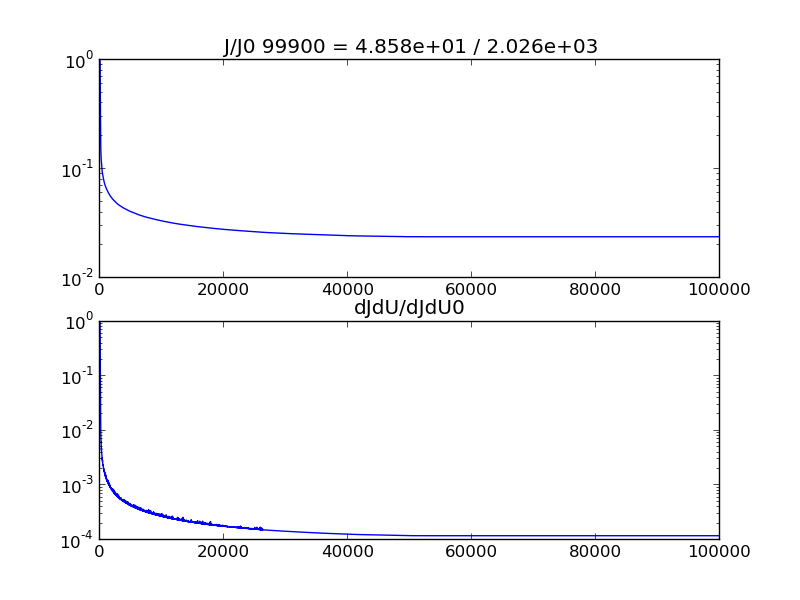}
 \caption{2D case with hole : the components of $A$ (left), $J$ and $G$ (right) ($\varepsilon = 0.05, \sigma = 0.1$).}
 \label{fig:qcqsingtrou05}
\end{figure}

\begin{figure}[h]
 \centering
 \includegraphics[width=5.5cm,bb=0 0 576 432,clip=true,trim=50pt 30pt 30pt 20pt]{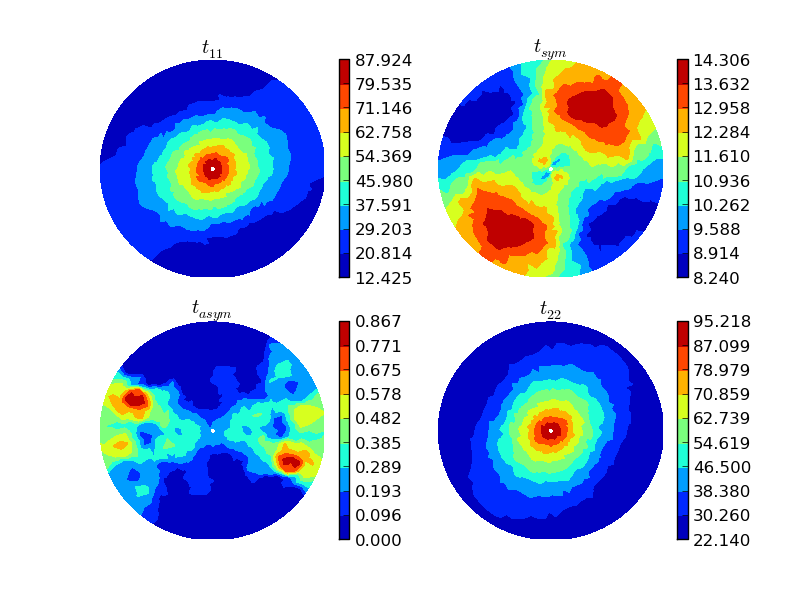}
 \includegraphics[width=5.5cm,bb=0 0 576 432,clip=true,trim=30pt 30pt 30pt 20pt]{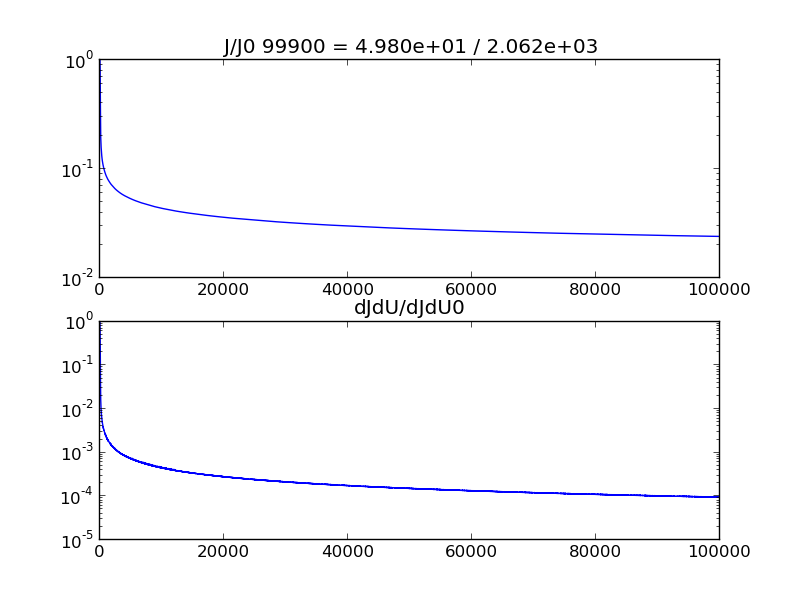}
 \caption{2D case with hole : the components of $A$ (left), $J$ and $G$ (right) ($\varepsilon = 0.025, \sigma = 0.1$).}
 \label{fig:qcqsingtrou025}
\end{figure}

\begin{figure}[h]
 \centering
 \begin{tabular}{ll}
 \includegraphics[width=1.95cm,bb=0 0 96 96,clip=true,trim=0pt 35pt 0pt 45pt]{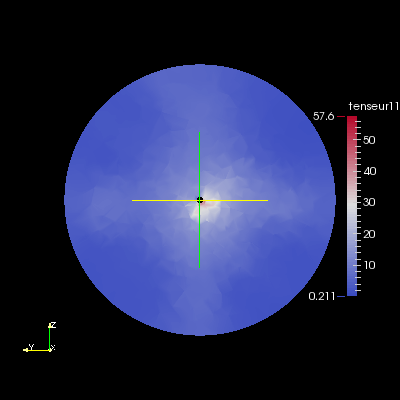}
 \includegraphics[width=1.95cm,bb=0 0 96 96,clip=true,trim=0pt 35pt 0pt 45pt]{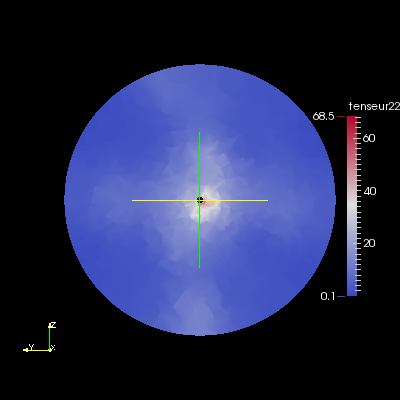}
 \includegraphics[width=1.95cm,bb=0 0 96 96,clip=true,trim=0pt 35pt 0pt 45pt]{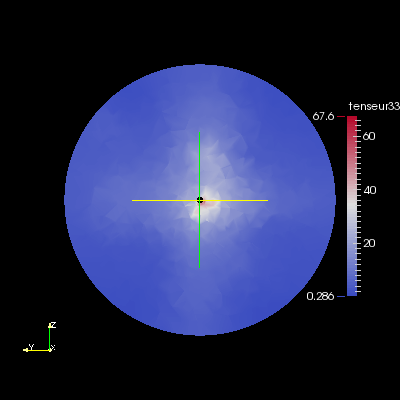}\\
 \includegraphics[width=1.95cm,bb=0 0 96 96,clip=true,trim=0pt 35pt 0pt 45pt]{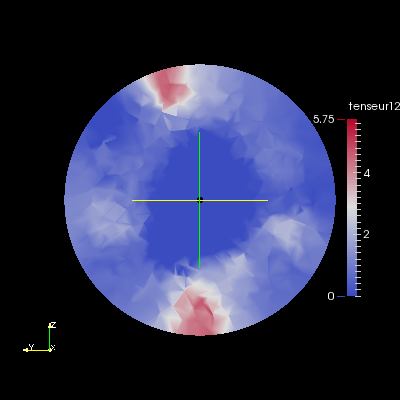}
 \includegraphics[width=1.95cm,bb=0 0 96 96,clip=true,trim=0pt 35pt 0pt 45pt]{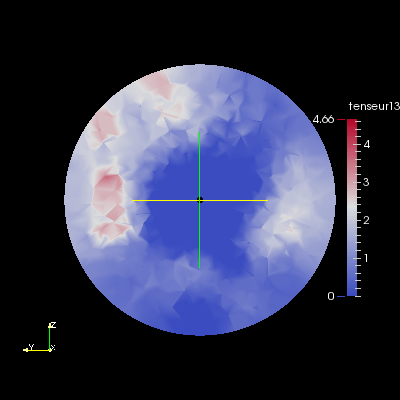}
 \includegraphics[width=1.95cm,bb=0 0 96 96,clip=true,trim=0pt 35pt 0pt 45pt]{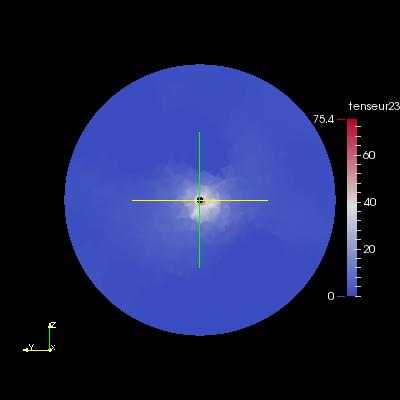}
 \includegraphics[width=5.5cm,clip=true,trim=30pt 217pt 30pt 20pt]{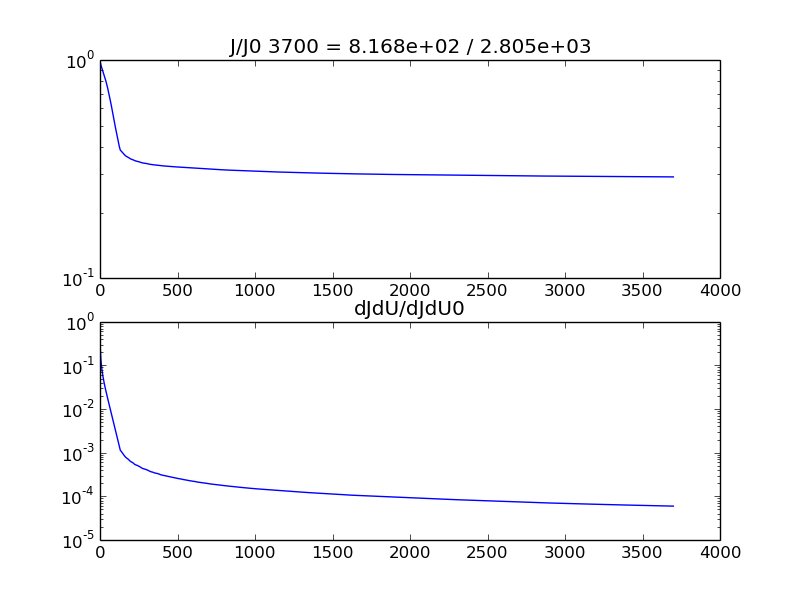}\\
 \includegraphics[width=1.95cm,bb=0 0 96 96,clip=true,trim=0pt 35pt 0pt 45pt]{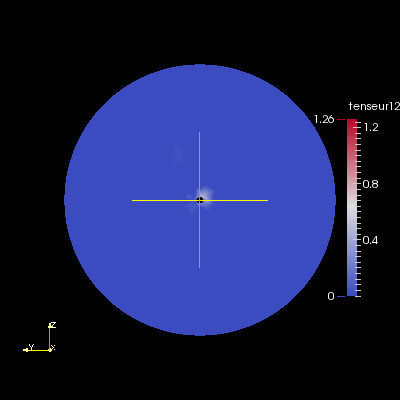}
 \includegraphics[width=1.95cm,bb=0 0 96 96,clip=true,trim=0pt 35pt 0pt 45pt]{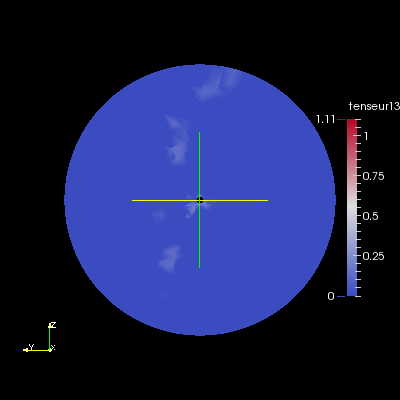}
 \includegraphics[width=1.95cm,bb=0 0 96 96,clip=true,trim=0pt 35pt 0pt 45pt]{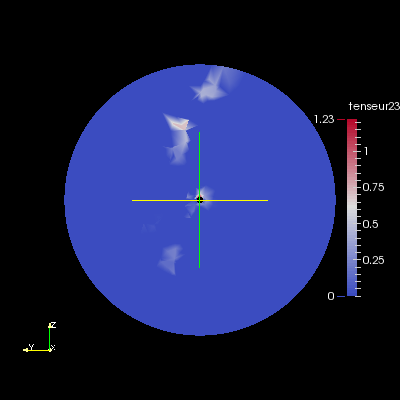}
 \includegraphics[width=5.5cm,clip=true,trim=30pt 28pt 30pt 216pt]{./Tests-step2_Trace_3D_3_qcq_sing_trou_V2_ref2_025_fin_J_dJdU.png}
 \end{tabular}
 \caption{3D with hole: (left) Visualisation of the components of the final control $A$ for the plane (0,Y,Z), $A_{11}, A_{22}, A_{33}$ (line 1), symmetric part of $A_{12}, A_{13}, A_{23}$  (line 2) and asymmetric part  of $A_{12}, A_{13}, A_{23}$ (line 3) ($\varepsilon = 0.05, \sigma = 0.1$). (right) Relatives evolutions of $J$ (line 2) and $||G||$ (line 3).} \label{fig:3D_with_hole}

\end{figure}

\section*{Acknowledgments} The authors gratefully acknowledge the support of le Conservatoire National des Arts et M\'{e}tiers (Paris, France) during the invitation of P. Kogut and the support of the French ANR Project CISIFS for an initiating paper on the subject.


\medskip
Received xxxx 20xx; revised xxxx 20xx.
\medskip

                \end{document}